\documentclass[12pt]{extarticle}
\usepackage[utf8]{inputenc}
\usepackage[T1]{fontenc}

\usepackage{caption}
\oddsidemargin \evensidemargin
\usepackage{amssymb}
\usepackage{amsmath}
\usepackage{comment}
\usepackage{amsthm}
\usepackage{amssymb}
\usepackage{algorithmicx}
\usepackage{algorithm}
\usepackage{algpseudocode}
\usepackage{mathtools}
\usepackage[margin=1cm]{caption}
\usepackage{subcaption}
\usepackage{mathrsfs}
\usepackage{xcolor}
\usepackage{tikz,tikz-3dplot,tikz-cd}
\usepackage{pgfplots}
\usepackage{url}

\usepackage{hyperref}
\hypersetup{
    colorlinks=true,
    linkcolor=orange!80!black,
    filecolor=magenta,      
    urlcolor=blue,
    citecolor=blue,
}

\definecolor{mycolor1}{rgb}{0.00000,0.44700,0.74100}
\definecolor{mycolor2}{rgb}{0.8500, 0.3250, 0.0980}
\definecolor{mycolor3}{rgb}{0.9290, 0.6940, 0.1250}
\definecolor{mycolor4}{rgb}{0.4940, 0.1840, 0.5560}
\definecolor{mycolor5}{rgb}{0.4660, 0.6740, 0.1880}

\newtheorem{definition}{Definition}[section]
\newtheorem{assumption}[definition]{Assumption}
\newtheorem{proposition}[definition]{Proposition}
\newtheorem{corollary}[definition]{Corollary}
\newtheorem{theorem}[definition]{Theorem}

\newtheorem{lemma}[definition]{Lemma}

\newtheorem*{theorem*}{Theorem}

\theoremstyle{remark}
\newtheorem{remark}[definition]{Remark}
\newtheorem{examplex}[definition]{Example}
\newenvironment{example}
  {\pushQED{\qed}\examplex}
  {\popQED\endexamplex}

\usepackage{mathtools} 
\mathtoolsset{showonlyrefs,showmanualtags}
\numberwithin{equation}{section}

\newcommand{\R}{\mathbb{R}}

\newcommand{\C}{\mathbb{C}}

\newcommand{\PP}{\mathbb{P}}
\newcommand{\Af}{\mathbb{A}}
\newcommand{\A}{\mathcal{A}}

\newcommand{\cA}{\mathcal{A}}

\title{\bf Adjoints and Canonical Forms of \\
Tree Amplituhedra}
\author{Kristian Ranestad, Rainer Sinn and Simon Telen}
\date{}

\begin{document}
\maketitle

\begin{abstract}
    \noindent We consider semi-algebraic subsets of the Grassmannian of lines in three-space called tree amplituhedra. These arise in the study of scattering amplitudes from particle physics. Our main result states that tree amplituhedra in ${\rm Gr}(2,4)$ are positive geometries. The numerator of their canonical form plays the role of the adjoint in Wachspress geometry, and is uniquely determined by explicit interpolation conditions.
\end{abstract}

\section{Introduction}

Amplituhedra are geometric objects arising in particle physics related to integral representations of scattering amplitudes. They were first defined by Arkani-Hamed and Trnka in \cite{arkani2014amplituhedron}. Such amplitudes play an important role in making predictions for particle scattering experiments. Though this is the main application, our paper assumes no background from physics. From a mathematical perspective, amplituhedra are interesting objects in their own right: they naturally generalize projective cyclic polytopes. We consider a particular family of amplituhedra which are semi-algebraic subsets of the real Grassmannian of lines in $\mathbb{R}\mathbb{P}^3$.

Conventionally, an amplituhedron depends on five parameters called $k, \ell, m, n, L$, see for instance \cite[Equation (6.36)]{arkani2017positive}. The values we consider are $k=m=2$, $L=0$ and $n$ is arbitary. When $L= 0$, the value of $\ell$ is irrelevant. The word ``tree'' in the title indicates the choice $L = 0$. In physics, this constrains the type of interactions allowed between particles. 

\begin{definition}[Amplituhedron] \label{def:amplituh}
Let $Z \in \R^{n \times 4}$ be a totally positive matrix, with $n \geq 4$. The \emph{$n$-point tree amplituhedron ${\cal A}_n$} is the image of the nonnegative Grassmannian ${\rm Gr}(2,n)_{\geq 0}$ under the rational map $Z : {\rm Gr}(2,n) \dashrightarrow {\rm Gr}(2,4)$, which takes a $2 \times n$ matrix $X$ to $X \cdot Z$.
\end{definition}

Here, $2 \times n$ matrices of rank $2$ represent lines in $\mathbb{RP}^{n-1}$. For brevity, we drop the words ``$n$-point'' and ``tree'' in our paper and call ${\cal A}_n$ simply the amplituhedron. To emphasize the dependence on $Z$, we sometimes write ${\cal A}_n(Z)$. Definitions of total positivity, nonnegative Grassmannians and more details on the map $Z$ are given in Section~\ref{sec:2}. 

\begin{remark} \label{rem:loopvstree}
The amplituhedron for $k = 2, m = 2, L = 0$ is equal to that for $k = 0, m = 4, \ell = 2, L = 1$ (which is called a \emph{one-loop amplituhedron}). Our choice of parameters is especially relevant in the physics application, see \cite[Section 6.6]{arkani2017positive}. 
\end{remark}

To see how Definition \ref{def:amplituh} generalizes polytopes, simply replace ${\rm Gr}(2,n)$ and ${\rm Gr}(2,4)$ by ${\rm Gr}(1,n) = \mathbb{RP}^{n-1}$ and ${\rm Gr}(1,4) = \mathbb{RP}^3$. The image of $(\mathbb{RP}^{n-1})_{\geq 0}$ under $Z$ is the convex hull of the rows of $Z$. Total positivity of $Z$ implies that the result is a cyclic polytope, see \cite{sturmfels1988totally}. 

 Amplituhedra have been studied intensively. Special attention has been given to certain subdivisions, called triangulations, of ${\cal A}_n$. These are strongly related to the positroid stratification of the nonnegative Grassmannian. Initial steps were taken in the original paper \cite{arkani2014amplituhedron}, and further efforts from the physics literature include \cite{arkani2018unwinding,franco2015anatomy}. Recently, these subdivisions were studied using the combinatorics of oriented matroid theory and plabic graphs \cite{lukowski2023positive,parisi2023m}. The case $m = 1$ is studied in \cite{mandelshtam2023combinatorics}, where the positivity constraint on the matrix~$Z$ is dropped.

Our interest in amplituhedra comes from a question in positive geometry \cite{arkani2017positive,lam2022invitation}. A positive geometry is a tuple $({\cal X},{\cal X}_{\geq 0})$ of an irreducible complex projective variety ${\cal X}$ and a semi-algebraic subset ${\cal X}_{\geq 0} \subset {\cal X}$. This semi-algebraic subset is characterized by the existence of a unique meromorphic top form on ${\cal X}$ determined recursively. That form is called the canonical form, and it is denoted by $\Omega({\cal X}_{\geq 0})$. We will recall the precise definition in Section \ref{sec:2}. It is known that polytopes in $\mathbb{P}^d$ are positive geometries, and so are positive parts of toric varieties and the nonnegative Grassmannian \cite{arkani2017positive}. Interestingly, the amplituhedron ${\cal A}_n$ inside the complex Grassmannian ${\rm Gr}_{\mathbb{C}}(2,4)$ has been one of the main inspirations for the definition of a positive geometry, but it was not known whether its canonical form exists (uniquely). The same is true for amplituhedra defined by other parameters. In \cite[Conjecture 2]{lam2022invitation}, it is explicitly conjectured that amplituhedra are positive geometries. Our main contribution is to settle this conjecture for the tree amplituhedra in Definition \ref{def:amplituh}.

\begin{theorem} \label{thm:main}
    For generic $Z$, the amplituhedron $({\rm Gr}_{\mathbb{C}}(2,4),{\cal A}_n)$ is a positive geometry. 
\end{theorem}

We study the amplituhedron and its boundaries from the perspective of semi-algebraic geometry. We describe the canonical form in terms of these boundaries. This is in the spirit of \cite{kohn2021adjoints}, which deals with positive geometries in the projective plane. The denominator of the canonical form is the product of defining polynomials of its codimension-one boundaries. The numerator is the defining polynomial $\alpha_{{\cal A}_n}$ of a 3-fold in ${\rm Gr}(2,4)$, called the adjoint 3-fold of ${\cal A}_n$. The reason for this name is that $\alpha_{{\cal A}_n}$ plays the role of the adjoint polynomial from Wachspress geometry \cite{gaetz,kohn2020projective}. 
It satisfies a concrete list of interpolation conditions, namely it vanishes on an arrangement of curves in ${\rm Gr}(2,4)$ defined by Schubert conditions called the residual arrangement. The need for the study of this residual arrangement was expressed explicitly by Lam in \cite{lam2022invitation} (bottom of page 15).
The interpolation conditions were found in \cite[Section 2.2]{arkani2015positive} as well. 
We deduce them from a full description of the face structure of ${\cal A}_n$ (Proposition \ref{prop:boundary faces}). We show that the residual arrangement
determines $\alpha_{{\cal A}_n}$ uniquely (Theorem~\ref{adjoint}). The adjoint gives the meromorphic top form  $\Omega({\cal A}_n)$ and we verify the axioms for a positive geometry from \cite{arkani2017positive} to prove Theorem \ref{thm:main}. We are optimistic that the techniques we use in this article can be used to prove the same result for general tree amplituhedra.

This paper is structured as follows. Section \ref{sec:2} contains preliminaries. We elaborate on Definition \ref{def:amplituh}. We recall an alternative definition in terms of inequalities on ${\rm Gr}(2,4)$ and state the definition of a positive geometry. Section \ref{sec:3} describes the algebraic boundary $\partial_a{\cal A}_n$ of ${\cal A}_n$. That is, the Zariski closure in ${\rm Gr}(2,4)$ of its Euclidean boundary. For lack of a precise reference, we include a proof of this description (Proposition \ref{prop:algboundary}). We also introduce a stratification of $\partial_a {\cal A}_n$ analogous to the face lattice of a polytope. This is used in Section \ref{sec:4} to identify the residual arrangement ${\cal R}({\cal A}_n)$. In Section \ref{sec:5}, we show that there exists a unique degree $(n-4)$-form in the coordinate ring of ${\rm Gr}(2,4) \subset \PP^5$ which interpolates ${\cal R}({\cal A}_n)$. We use this to find the canonical form $\Omega({\cal A}_n)$ in Section \ref{sec:6}, and prove Theorem~\ref{thm:main}. 

\section{Amplituhedra and positive geometries} \label{sec:2}
We start with a clarification of the concepts and notation in Definition \ref{def:amplituh}. The real Grassmannian ${\rm Gr}_{\mathbb{R}}(k,m)$ is the space of $k$-dimensional $\R$-vector spaces in $\R^m$. Equivalently, it is the space of projective $(k-1)$-planes in the real projective space $\mathbb{R}\PP^{m-1}$. To simplify the notation, we will drop $\mathbb{R}$ and work implicitly over the real numbers, unless indicated otherwise. The Grassmannian is a projective variety via its Pl\"ucker embedding ${\rm Gr}(k,m) \hookrightarrow \PP^{\binom{m}{k} - 1}$. This sends $V \in {\rm Gr}(k,m)$ to the tuple of $k \times k$ minors of a $k \times m$ matrix with row span $V$. 

\begin{example}[$k =2, m = 4$] \label{ex:G24}
The Grassmannian ${\rm Gr}(2,4)$ consists of all lines in projective 3-space. The line through the points $(x_1:x_2:x_3:x_4)$ and $(y_1:y_2:y_3:y_4)$ in $\mathbb{P}^3$ is represented by the $2 \times 4$ matrix $X = \begin{pmatrix} x_1 & x_2 & x_3 & x_4 \\ y_1 & y_2 & y_3& y_4 \end{pmatrix}$. The Pl\"ucker embedding is given by 
\begin{equation} \label{eq:pluckemb}
\begin{pmatrix} x_1 & x_2 & x_3 & x_4 \\ y_1 & y_2 & y_3& y_4 \end{pmatrix} \longmapsto (\, x_iy_j-x_jy_i \, )_{1 \leq i < j \leq 4} \quad \in \PP^5. 
\end{equation}
The homogeneous coordinate indexed by $(i,j)$ on $\PP^5$ is denoted by $p_{ij}$, and \eqref{eq:pluckemb} identifies ${\rm Gr}(2,4)$ with the quadratic hypersurface 
\begin{equation} \label{eq:plucker}
{\rm Gr}(2,4) = \{ (p_{ij})_{1 \leq i < j \leq 4} \in \PP^5 ~|~ p_{12}p_{34} - p_{13}p_{24} + p_{14}p_{23} = 0 \}. 
\end{equation}
The quadratic equation in \eqref{eq:plucker} is called the \emph{Pl\"ucker relation} for ${\rm Gr}(2,4)$. 
\end{example}
Generalizing Example \ref{ex:G24}, the \emph{Pl\"ucker coordinates} on $\PP^{\binom{m}{k} -1}$ are denoted by $p_I$, where $I$ ranges over all $k$-element subsets of $[m] = \{1, \ldots, m \}$. A $k$-dimensional vector space $V \subset \R^m$ is called \emph{totally nonnegative} (resp.~\emph{totally positive}) if all of its Pl\"ucker coordinates satisfy $p_I(V) \geq 0$ (resp.~$p_I(V) > 0$). Of course, the sign of the Pl\"ucker coordinates is only defined up to projective scaling, so here `$p_I(V) \geq 0$ for all $I$' is short for `all (nonzero) Pl\"ucker coordinates have the same sign'. Since the sign of a real homogeneous polynomial of even degree is well-defined in projective space, we can write these inequalities as $p_I(V) p_J(V)\geq 0$ for all pairs of $k$-element subsets $I,J\subset [m]$.
The set of all totally nonnegative points in ${\rm Gr}(k,m)$ is the \emph{totally nonnegative Grassmannian} ${\rm Gr}(k,m)_{\geq 0}$. The \emph{totally positive Grassmannian} ${\rm Gr}(k,m)_{>0}$ is defined analogously. A $k \times m$ matrix $X$ is called \emph{totally nonnegative} (resp.~\emph{totally positive}) if its row space is totally nonnegative (resp.~totally positive) in ${\rm Gr}(k,m)$. Below, we drop the word totally in this terminology for convenience.

Let $Z \in \R^{n \times m}$ be a positive (i.e., totally positive) matrix, with $n \geq m$. This defines a linear map $X \mapsto X \cdot Z$ sending $k \times n$ matrices to $k \times m$ matrices. It is easy to see that this map is well-defined modulo the action of ${\rm GL}_k(\R)$ on $k \times n$ and $k \times m$ matrices, so that it gives a rational map $Z: {\rm Gr}(k,n) \dashrightarrow {\rm Gr}(k,m)$. This is well-defined on $k$-dimensional vector spaces in $\R^n$ represented by a $k \times n$ matrix $X$ such that $X \cdot Z$ has rank $k$. It turns out that if $X$ is nonnegative, positivity of $Z$ ensures that $X \cdot Z$ has rank $k$ \cite[Sec.~4]{arkani2014amplituhedron}. Therefore, the image $Z({\rm Gr}(k,n)_{\geq 0}) \subset {\rm Gr}(k,m)$ of the nonnegative Grassmannian is well-defined. For $k = 2, m = 4$, this is our amplituhedron ${\cal A}_n$ from Definition \ref{def:amplituh}.

We now recall a different, equivalent definition of ${\cal A}_n$ in terms of inequalities on ${\rm Gr}(2,4)$. This first appeared in \cite{arkani2018unwinding}, and was proved for our amplituhedra in \cite[Theorem 5.1]{parisi2023m}. We think of the rows of our matrix $Z \in \R^{n \times 4}$ as points $Z_1, \ldots, Z_n$ in $\PP^3$. It is customary to denote the line in ${\rm Gr}(2,4)$ through points $A, B \in \PP^3$, $A \neq B$, by $AB$. Write $p_{12}, \ldots, p_{34}$ for the Pl\"ucker coordinates of this line and denote the Pl\"ucker coordinates of $Z_iZ_j$ by 
\[ Z_iZ_j \, = \, (q^{ij}_{12}: q^{ij}_{13}: q^{ij}_{14} : q^{ij}_{23} : q^{ij}_{24} : q^{ij}_{34}) \, \in \PP^5.\]
Note that the positivity condition on $Z$ implies $Z_i \neq Z_j$. Requiring that the line $AB$ meets the line $Z_iZ_j$ gives a linear condition on the Pl\"ucker coordinates of $AB$, namely
\begin{equation} \label{eq:ABij} 
\langle ABij \rangle \, = \, p_{12} \, q^{ij}_{34} - p_{13}  \, q^{ij}_{24} + p_{14} \, q^{ij}_{23} + p_{23} \, q^{ij}_{14} - p_{24} \, q^{ij}_{13} + p_{34} \, q^{ij}_{12} \, = \, 0,
\end{equation}
which is the determinant of the $4\times 4$ matrix with columns $A$, $B$, $Z_i$, and $Z_j$. The amplituhedron ${\cal A}_n$ can be defined in terms of sign conditions on these linear expressions. 

\begin{definition} \label{def:amplituhedron2}
Let $Z \in \R^{n \times 4}$ be a totally positive matrix. The subset ${\cal F}^\circ_n(Z)$ is the semi-algebraic subset of ${\rm Gr}(2,4)$ consisting of points $(p_{12}: \cdots: p_{34})$ at which
\begin{enumerate}
    \item all linear forms $\langle ABi(i+1) \rangle, i = 1, \ldots, n-1$ and $\langle AB 1n \rangle$ from 
    \eqref{eq:ABij} are nonzero and have the same~sign,
    \item the sequence $(  \langle AB12 \rangle  , \,  \langle AB13 \rangle  , \, \ldots , \, \langle AB1n \rangle )$ has exactly two sign flips, ignoring zeros.
\end{enumerate}
\end{definition}
This Definition \ref{def:amplituhedron2} gives an alternative, inequality description of the amplituhedron. To our knowledge, equivalence of the two definitions was first shown in {\cite[Theorem 5.1]{parisi2023m}}. The notation ${\cal F}^\circ_n(Z)$ is inspired by that theorem, and the next lemma follows from its proof. 
\begin{lemma} \label{lem:equivdef}
The \emph{amplituhedron} ${\cal A}_n(Z)$ is the closure of ${\cal F}_n^\circ(Z)$ in ${\rm Gr}(2,4)$. More precisely,
\[ Z({\rm Gr}(2,n)_{>0}) \, \subset \, {\cal F}^\circ_n(Z) \, \subset \, {\cal A}_n(Z),\]
and all inclusions turn into equalities when taking the closure in the Euclidean topology.
\end{lemma}

\begin{example}[$n=4$]
All lines $AB$ in the open subset ${\cal F}^{\circ}_4(Z)$ of the amplituhedron ${\cal A}_4$ satisfy 
\begin{equation}  \label{eq:onlyfirstineq}\langle AB12 \rangle > 0, \quad \langle AB23 \rangle > 0, \quad \langle AB34 \rangle > 0, \quad  \langle AB14 \rangle > 0. 
\end{equation}
This is point 1 in Definition \ref{def:amplituhedron2}. Together with the Pl\"ucker relation \eqref{eq:plucker}, this ensures that $\langle AB13 \rangle$ and $\langle AB24 \rangle$ have the same sign. Hence, the inequalities \eqref{eq:onlyfirstineq} define an open semi-algebraic set with two connected components, each corresponding to a sign choice for $\langle AB13 \rangle$ and $\langle AB24 \rangle$. Since ${\rm Gr}(2,n)_{>0}$ is an open ball \cite{galashin2022totally}, $Z({\rm Gr}(2,n)_{>0})$ is connected by Lemma \ref{lem:equivdef}. Point 2 in Definition \ref{def:amplituhedron2} fixes the signs of $\langle AB13 \rangle$ and $\langle AB24 \rangle$ to be `$-$', and selects one of two connected components.
\end{example}

The aim of the rest of this section is to recall the definition of a positive geometry, as introduced in \cite{arkani2017positive}. A nice overview of the state of the art on this topic is found in \cite{lam2022invitation}.

Let ${\cal X}$ be a $d$-dimensional irreducible complex projective variety, defined over $\mathbb{R}$. Let ${\cal X}_{\geq 0}$ be a $d$-dimensional semi-algebraic subset of the real points ${\cal X}(\mathbb{R})$ of ${\cal X}$. We require that the interior ${\cal X}_{>0}$ of ${\cal X}_{\geq 0}$ (with respect to the Euclidean topology on ${\cal X}(\mathbb{R})$) is an oriented manifold, such that its closure is ${\cal X}_{\geq 0}$. The algebraic boundary $\partial_a {\cal X}_{\geq 0}$ is the Zariski closure in ${\cal X}$ of ${\cal X}_{\geq 0} \setminus {\cal X}_{>0}$. The irreducible components of $\partial_a {\cal X}_{\geq 0}$ are prime divisors $D_1, \ldots, D_r$ on ${\cal X}$ (see \cite[Lemma 3.2(a)]{so2}). We define $(D_i)_{\geq 0}$ as the closure in $D_i(\mathbb{R})$ of the interior of~$D_i \cap {\cal X}_{\geq 0}$. 
\begin{definition} \label{def:posgeom}
    The pair $({\cal X}, {\cal X}_{\geq 0})$ is called a \emph{positive geometry} if there exists a unique meromorphic $d$-form $\Omega({\cal X}_{\geq 0})$ on ${\cal X}$, called \emph{canonical form}, satisfying the following axioms.
    \begin{itemize}
        \item If $d > 0$, $\Omega({\cal X}_{\geq 0})$ has poles only along $D_1, \ldots, D_r$. Moreover, for each algebraic boundary component $D_i$, $\Omega({\cal X}_{\geq 0})$ has a simple pole along $D_i$ and the Poincar\'e residue ${\rm Res}_{D_i}\Omega({\cal X}_{\geq 0})$ equals the canonical form of the positive geometry $(D_i, (D_i)_{\geq 0})$. Here $(D_i)_{>0}$ is an oriented manifold. Its  orientation is induced by that of ${\cal X}_{>0}$.
        \item If $d = 0$, ${\cal X} = {\cal X}_{\geq 0}$ is a point and $\Omega({\cal X}_{\geq 0}) = \pm 1$. The sign is the orientation. 
    \end{itemize}
\end{definition}

The amplituhedron ${\cal A}_n \subset {\rm Gr}_{\mathbb{C}}(2,4)$ was conjectured to be a positive geometry, see for instance  \cite[Conjecture 2]{lam2022invitation}. Polytopes in projective space are known to be positive geometries. Their canonical form provides a recipe for constructing canonical forms more generally.
\begin{example}[${\cal X} = \mathbb{CP}^1$] \label{ex:canformlinesegment} A line segment $[a,b]\subset \R = \{ (x,1)\colon x\in \R\} \subset \mathbb{RP}^1$ defines the positive geometry $(\C\PP^1,[a,b])$ in the complex projective line. The canonical form $\Omega([a,b])$ is given, in the local coordinate $x$, by 
\[ \Omega ([a,b]) \, = \, \frac{b-a}{(x-a)(b-x)} \, {\rm d} x. \]
This form has poles only along the boundary points of our line segment. From the identity
\[ \Omega([a,b]) \, = \, \left (\frac{b-a}{b-x}  \right ) \, \frac{{\rm d}(x-a)}{x-a}\]
we see that ${\rm Res}_{x=a}\Omega([a,b]) = 1$. Indeed, recall that the residue is obtained by substituting $x = a$ in the expression between parentheses. Similarly, ${\rm Res}_{x=b}\Omega([a,b]) = -1$.
\end{example}
\begin{example}[${\cal X} = \mathbb{CP}^2$] \label{ex:canformpentagon}
   Consider the pentagon $P$ in $\R^2$ with vertices $(1,0)$, $(3,0)$, $(4,2)$, $(2,4)$, and $(0,2)$.  Via $\R^2 = \{(x,y,1)\colon x,y\in \R\} \subset \R\PP^2$, we view this as a semi-algebraic subset of $\mathbb{CP}^2$. We claim that $(\C\PP^2,P)$ is a positive geometry with canonical form
   \[ \Omega(P) \, = \, \frac{ -4 \cdot (-4x^2 +16x  - 2y^2 + 28y + 48)}{(2x + y - 2)\cdot( x - y + 2)\cdot  y \cdot( -2x + y + 6)\cdot( -x - y + 6)} \, \, {\rm d}x \wedge {\rm d} y.\]
   The denominator of $\Omega(P)$ is the product of the lines containing the edges of $P$ and hence has the right poles. To determine the residue along $y = 0$ in the coordinate $x$, we first switch ${\rm d}x \wedge {\rm d}y$ to $-{\rm d}y \wedge {\rm d}x$, then drop ${\rm d}y$ and the factor $y$ in the denominator, and finally substitute $y = 0$. The result is 
   \[ {\rm Res}_{y=0}\, \Omega(P) \, = \, \frac{4 \cdot(-4x^2 + 16x + 48)}{(2x-2)(x+2)(-2x+6)(-x+6)} \, \frac{{\rm d} x}{x} \, = \, \frac{2}{(x-1)(3-x)} \, \frac{{\rm d} x}{x}.\]
   We have seen in Example \ref{ex:canformlinesegment} that this is indeed the canonical form of the edge $\{y=0\} \cap P$ in its Zariski closure. Notice that the numerator of $\Omega(P)$ is such that, when restricted to any of the lines in the algebraic boundary of $P$, it cancels two poles which are not vertices (here $x+2$ and $-x+6$). This numerator is called the \emph{adjoint} of $P$. Its cancellation property implies that the adjoint interpolates the \emph{residual arrangement} of $P$. For any simple polytope $P$, this is defined as the union of all intersections of boundary components which do not intersect $P$ \cite{kohn2020projective}. For a polygon, the residual arrangement consists of points. Figure \ref{fig:adjpentagon} illustrates this in our example. The residual arrangement is in blue, the adjoint conic is in red. 
    \begin{figure}
        \centering
        \includegraphics[height = 6cm]{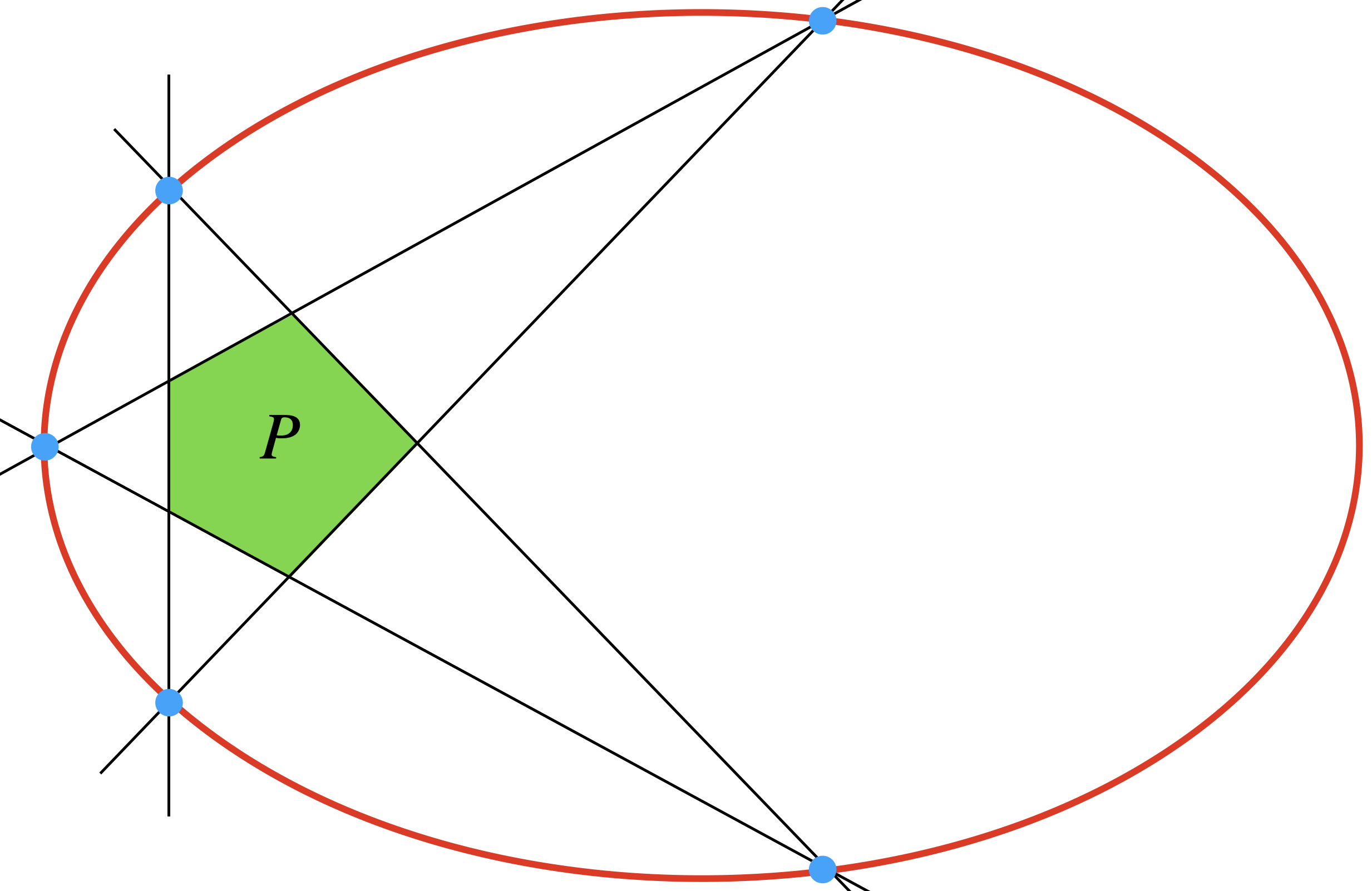}
        \caption{The adjoint of a pentagon is a conic through five residual points }
        \label{fig:adjpentagon}
    \end{figure}
   \end{example}
Example \ref{ex:canformpentagon} illustrates the importance of the algebraic boundary and the adjoint hypersurface in the study of positive geometries. Our strategy for writing down the canonical form of $\Omega({\cal A}_n)$ rests on this. We identify and stratify the algebraic boundary first (Section \ref{sec:3}), and then switch to the residual arrangement and the adjoint (Sections \ref{sec:4} and \ref{sec:5}). 

\begin{remark} \label{rem:sign}
We want to point out that the sign of the residue at a boundary vertex of ${\cal X}_{\geq 0}$ depends on the order in which the iterated residues are taken. One can verify this in the simple example $\Omega(\mathbb{R}^2_{\geq 0}) \, = \, \frac{1}{xy} \, {\rm d} x \wedge {\rm d}y$, for the point $(0,0)$. 
\end{remark}

\section{Algebraic boundary} \label{sec:3}
The \emph{boundary} $\partial {\cal A}_n$ of the amplituhedron ${\cal A}_n$ is ${\cal A}_n \setminus {\rm int}({\cal A}_n)$, where ${\rm int}(\cdot)$ takes the interior with respect to the Euclidean topology on ${\rm Gr}_{\mathbb{R}}(2,4) \subset \mathbb{P}^5$. The \emph{algebraic boundary} $\partial_a {\cal A}_n$ is defined as the Zariski closure in ${\rm Gr}_{\mathbb{C}}(2,4)$ of the boundary $\partial {\cal A}_n$. Though its description appears to be well-known and widely used, see for instance \cite{arkani2017positive,arkani2015positive}, we could not find a precise reference. We include the following statement with a self-contained proof for completeness. 

\begin{proposition} \label{prop:algboundary}
    The algebraic boundary $\partial_a{\cal A}_n$ of ${\cal A}_n$ $(n \geq 4)$ is the union of the following hyperplane sections of ${\rm Gr}_{\mathbb{C}}(2,4)$: $\langle ABi(i+1) \rangle = 0, \, i = 1, \ldots, n-1$, and $\langle AB1n \rangle = 0$. 
\end{proposition}

Introducing the ``cyclic'' notation $\langle ABn(n+1) \rangle = \langle AB1n \rangle$ allows to simplify the notation slightly. The algebraic boundary components are then given by $\langle ABi(i+1) \rangle = 0, \, i \in [n]$. Our proof of Proposition \ref{prop:algboundary} uses the following lemma.

\begin{lemma} \label{lem:projaway}
    Let $AB \in {\rm Gr}(2,4)$ be a real line in $\PP^3$ and let $Y \in \mathbb{R}^{2 \times 4}$ be a matrix representing $AB$. The \emph{projection away from $AB$} is represented by a matrix $Y^\perp \in \mathbb{R}^{2 \times 4}$ whose row span is the orthogonal complement of the row span of $Y$. Let $z_i = Y^\perp Z_i \in \mathbb{R}^2$, where $Z_i \in \mathbb{R}^4$ is the $i$-th row of $Z$, and let $M = (z_1~z_2~\cdots~z_n) \in \mathbb{R}^{2 \times n}$. There is a nonzero constant $c \in \mathbb{R}$ such that for all $1 \leq i < j \leq n$ the $(i,j)$ minor of $M$ is \[ \det(M_{ij}) = \det(z_i~z_j) = c \cdot \langle ABij \rangle.\]
\end{lemma}
\begin{proof}
Let $Y^\dag \in \mathbb{R}^{2 \times 4}$ be the pseudo-inverse of $Y^\top$. The lemma follows from the identity
\[ \begin{pmatrix}
    Y^\dag \\ 
    Y^\perp
\end{pmatrix} \cdot \begin{pmatrix}
    Y^\top ~ Z_i ~  Z_j
\end{pmatrix} \, =  \, \begin{pmatrix}
    {\rm id}_{2 \times 2} & Y^\dag(Z_i~Z_j) \\ 
    0 & M_{ij}
\end{pmatrix}. \qedhere \]
\end{proof}

This lemma helps to translate the inequality description of the amplituhedron in Definition \ref{def:amplituhedron2} into conditions on the points $z_1, \ldots, z_n$ in the plane. We may fix signs so that $\langle AB12 \rangle > 0$ and $c > 0$. Then condition 1 in Definition \ref{def:amplituhedron2}, i.e.,~$\det(z_i~z_{i+1})>0$, translates to ``the angle $\arg(z_{i+1}) - \arg(z_i)$ between $z_i$ and $z_{i+1}$ lies in the open interval $(0,\pi)$ for $i = 1, \ldots, n-1$'' and similarly,  for $i=n$, we get $\arg(z_n) - \arg(z_1) \in (0,\pi)$. Here $\arg(z_i)$ is the angular coordinate of $z_i$. This means that the (piecewise linear) path connecting $z_1, z_2, \ldots, z_n$ revolves counter-clockwise around $(0,0)$ in $\mathbb{R}^2$. The two sign flips in condition 2 in Definition \ref{def:amplituhedron2} mean that this path makes a total angle $\phi$ satisfying $2 \pi<\phi <3 \pi$.

\begin{proof}[Proof of Proposition \ref{prop:algboundary}]
The nonnegative Grassmannian ${\rm Gr}(k,m)_{\geq 0}$ is homeomorphic to a $k(m-k)$-dimensional closed ball \cite{galashin2022totally}, hence it is a regular semi-algebraic set. The amplituhedron ${\cal A}_n\subset {\rm Gr}(2,4)$ is the image of ${\rm Gr}(k,m)_{\geq 0}$, so it is also regular. This implies that its algebraic boundary is a Weyl divisor in ${\rm Gr}_{\mathbb{C}}(2,4)$ by \cite[Lemma 3.2(a)]{so2}. 
From the inequality description of the amplituhedron in Definition \ref{def:amplituhedron2}, it follows that the hyperplane sections $\langle ABi(i+1) \rangle = 0, \, i \in [n]$, together with $\langle AB1j \rangle =0$ for $j\in [n]$ are the only candidates for the irreducible components of ${\cal A}_n$.  Indeed, they are the only functions appearing in the inequality description, so at least one of them has to be zero in a boundary point. 
Since ${\cal A}_n$ is a regular semi-algebraic set, such a hyperplane section $\langle AB ij \rangle =0$ is an irreducible component of $\partial_a {\cal A}_n$ if and only if there is a boundary point $AB$ of ${\cal A}_n$ for which we have $\langle AB ij \rangle =0$ and no other candidate divisor contains $AB$. We find such a boundary point for every hyperplane section listed in the proposition and show that no such point exists for $\langle AB1j \rangle $, $j\in \{3, \ldots, n-1\}$. 

It is convenient to define an operator $\sigma: \mathbb{R}^{2 \times n} \rightarrow \mathbb{R}^{2 \times n}$ which modifies a matrix $X$ by applying a circular shift to its columns, and then inverts the sign of the first column: 
    \[ \sigma \begin{pmatrix}
        x_1 & x_2 & \cdots & x_n \\
        y_1 & y_2 & \cdots & y_n
    \end{pmatrix}  \, = \, \begin{pmatrix}
        -x_n & x_1 & \cdots & x_{n-1}\\
        -y_n & y_1 & \cdots & y_{n-1}
    \end{pmatrix} .\]
    We write $\sigma^k = \sigma \circ \cdots \circ \sigma$ for the operator which applies $\sigma$ $k$ times, and $\sigma^0$ is the identity. 

    To identify a point on ${\cal A}_n \cap \langle AB 12 \rangle$ on which no other candidate bracket vanishes, we pick a general element $X$ in the following $3$-dimensional family of nonnegative matrices: 

    \[ \left \{ \begin{pmatrix}  1& x_1 & 0 & 0 & 0 & \cdots & 0 \\ 
    0 & y_1 & y_2 & 1  & 0 & \cdots & 0 \end{pmatrix}
     \, : \, x_1, y_1, y_2 \in \mathbb{R}_{>0} \right \} 
    \]
    We then send $X$ through the amplituhedron map: $X \cdot Z = AB$. Geometrically, $AB$ is an element of the $3$-dimensional set of lines passing through the line segment ${\rm conv}(Z_1,Z_2)$ and through the triangle ${\rm conv}(Z_2,Z_3,Z_4)$. Any element in that set satisfies $\langle AB12 \rangle = 0$, and it is straightforward to verify that for a general element $\langle AB ij \rangle \neq 0$ for other indices $i,j$. To find a line $AB$ satisfying $\langle ABi(i+1) \rangle$ and all other $\langle ABij \rangle$ nonzero, simply consider~$\sigma^{i-1}(X)$. 

    We have now shown that $\langle AB i (i+1) \rangle$ belongs to the algebraic boundary of the amplituhedron for $i \in [n]$ and switch to showing that $\langle AB1j \rangle = 0$, $j\in \{3, \ldots, n-1\}$ does not. 
    Fix $j \in \{3, \ldots, n-1\}$ and suppose that $AB \in {\cal A}_n$ is a point in the amplituhedron for which $\langle AB1j \rangle = 0$, all of the brackets $\langle AB i(i+1) \rangle$ for $i \in [n]$ are strictly positive and $\langle AB1j'\rangle \neq 0$ for $j' \neq j$. In the notation of Lemma \ref{lem:projaway}, this means in particular that $\arg(z_{i+1})-\arg(z_i) \in (0,\pi)$. Our strategy is to show that $\langle AB1(j-1) \rangle$ and $\langle AB1(j+1) \rangle$ have opposite sign, so that the sign flip condition 2 in Definition \ref{def:amplituhedron2} is satisfied for either sign of $\langle AB1j \rangle$, which means that $AB$ is an interior point. Assume first that $\langle AB1(j-1)\rangle  > 0$. For the projections, this entails $\arg(z_{j-1}) - \arg (z_1) \in (0,\pi)$.  Since $\langle AB 1j \rangle=0$, the line $z_1z_j$ passes through the origin in $\mathbb{R}^2$, hence $\arg(z_j) - \arg(z_1) \in \{0, \pi\}$. If this difference is $0$, then 
   \[ \arg(z_j) - \arg(z_{j-1}) = \arg(z_j) - \arg(z_1) + \arg(z_1) - \arg(z_{j-1}) = - (\arg(z_{j-1}) - \arg(z_1)) \in (-\pi,0),\]
   which contradicts our assumption that all $\langle ABi(i+1) \rangle$ are positive. Therefore, we have $\arg(z_j) - \arg(z_1) = \pi$, which implies 
   \[ \arg(z_{j+1}) - \arg(z_1) = \arg(z_{j+1}) - \arg(z_j) + \arg(z_j) - \arg(z_1) \in (-\pi,0).\] 
   So we have $\langle AB1(j+1) \rangle <0 $ showing that $AB$ is an interior point of $\cA_n$. Symmetrically, if $\langle AB1(j-1) \rangle < 0$, then we must have $\arg(z_j) - \arg(z_1) = 0$, which implies $\arg(z_{j+1}) - \arg(z_1) \in (0,\pi)$, so that $\langle AB1(j+1) \rangle > 0$. Again, the line $AB$ is an interior point of $\cA_n$. 
\end{proof}

Next, we describe the combinatorial structure of $\partial_a {\cal A}_n$. That is, we describe the irreducible components of all intersections between boundary components, and their inclusion relations. Our analysis needs the following generality assumption. 
\begin{assumption} \label{assum:general}
    There is no line $AB$ for which more than $4$ of $\langle ABi(i+1)\rangle,i \in [n] $ vanish.
\end{assumption}

\begin{remark}
    Assumption \ref{assum:general} holds for generic positive matrices $Z$, even though the lines $Z_iZ_{i+1}$ are not in generic position: they form a cycle in $\mathbb{P}^3$. There are precisely two lines meeting any four of these lines. These appear in Table \ref{tab:complexstrata} and Figure \ref{fig:boundarystrata} below. Generically these lines meet none of the other lines in the cycle.
\end{remark}

The components $\{ AB \in {\rm Gr}_\C(2,4) ~|~ \langle AB i(i+1) \rangle = 0 \}$ are tangent hyperplane sections of ${\rm Gr}_\C(2,4)$. Each one is defined by the following Schubert condition: 
$$L_{i(i+1)} : AB\cap Z_iZ_{i+1}\not=\emptyset.$$ 
The irreducible components of their intersections are conveniently described by products of Schubert conditions, which involve incidence conditions with lines, points and planes: 
\[L_{ij}: AB\cap Z_iZ_j\not=\emptyset, \quad  \, \, \, V_i: Z_i\in AB, \quad  \, \,\,  P_{(i-1)i(i+1)}: AB\subset Z_{i-1}Z_iZ_{i+1}. \]
This reads, from left to right, as ``lines intersecting the line $Z_iZ_j$'' (here $i < j$), ``lines containing the point $Z_i$'' and ``lines contained in the plane $Z_{i-1}Z_iZ_{i+1}$''. 
While the first condition $L_{i,j}$ defines a $3$-dimensional intersection of ${\rm Gr}_{\mathbb{C}}(2,4)$ with a tangent hyperplane, the latter two describe planes in ${\rm Gr}_\C(2,4)$. It is convenient to interpret the indices cyclically: $Z_{n+j} = Z_j$. Since we will often intersect different Schubert conditions, we omit the intersection symbol $\cap$ in the notation for brevity. For instance, $V_nP_{n(n+1)(n+2)}$ is the one-dimensional locus of lines through $Z_n$ \emph{and} contained in the plane $Z_1Z_2Z_n$. This explains the notation in Table \ref{tab:complexstrata}.

We stratify the algebraic boundary of ${\cal A}_n$ by intersections of its irreducible components. Here is a list of the irreducible components obtained by this process.
\begin{theorem} \label{thm:stratanew}
 Under Assumption \ref{assum:general}, the algebraic boundary $\partial_a {\cal A}_n$ of the amplituhedron is stratified into closed irreducible strata coming in 14 types. The strata are listed in Table \ref{tab:complexstrata}. The number of types in dimension $3$, $2$, $1$, $0$ is $1$, $3$, $4$, $6$ respectively. The number of strata of each type~is
\[ \begin{array}{c||c||c|c|c||c|c|c|c}
    \text{type} & (3,{\rm I}) & (2,{\rm I}) &
      (2, {\rm II}) &(2, {\rm III}) & (1,{\rm I})& (1,{\rm II}) & (1,{\rm III})& (1,{\rm IV}) \\  \hline
      \text{\# strata} &
      n & n &  n & \binom{n}{2}-n  & n& n(n-2)&n(n-4)&\binom{n}{3}-n(n-3)\\
\end{array} 
\]
\[
\begin{array}{c||c|c|c|c|c|c}
    {\rm type} & (0,{\rm I}) & (0,{\rm II}) &
      (0, {\rm III}) & (0,{\rm IV})& (0,{\rm V})&
      (0, {\rm VI}) \\  \hline
      \text{\# strata} &
     \binom{n}{2} & \frac{n}{2}(n-5)&n(n-4)&n\binom{n-3}{2}&n\binom{n-5}{2}& \frac{n}{24} (n-5)(n-6)(n-7) 
\end{array}.
\]
\end{theorem}
\begin{proof}
The $3$-dimensional strata are described in Proposition \ref{prop:algboundary}. Their intersections
\begin{equation}\label{schubert relation}
 L_{(i-1)i}L_{i(i+1)}\, = \, V_i\cup P_{(i-1)i(i+1)}, \,  i = 1, \ldots, n \quad \text{and} \quad  L_{i(i+1)}L_{j(j+1)}, \,  j>i+1
\end{equation}
form the $2$-dimensional strata. The first identity in the above display is the Schubert relation expressing the fact that the intersection between ${\rm Gr}_{\mathbb{C}}(2,4)$ and two consecutive tangent hyperplanes is the union of two planes.
This explains dimension three and two in Table \ref{tab:complexstrata}.

Since $L_{ij}$ defines a hyperplane section of ${\rm Gr}_{\mathbb{C}}(2,4)$, it has codimension one; $V_i$ and  $P_{(i-1)i(i+1)}$ have codimension two.

In dimension one, we have the product relation $V_iL_{i-1,i+1}=V_iP_{i-1,i,i+1}$.
Such a product has codimension three; it defines the line of intersection of two boundary planes. This line is called $(1,{\rm I},i)$ in the table.  All other $1$-dimensional strata lie in the intersection of three $3$-dimensional boundary components. The different types are easily described as products of Schubert conditions  
listed above whose codimensions add up to three.  Finally, the $0$-dimensional strata, i.e., the vertices of the algebraic boundary, are described as products of the listed Schubert conditions whose codimensions add up to four.   

All types of strata are illustrated in Figure \ref{fig:boundarystrata} for $n = 8$. Green triangles represent plane conditions $P_{ijk}$, red ticks are line conditions $L_{ij}$, and blue squares are point conditions $V_i$.

The count of strata of each type is now combinatorial. 
For instance, for type $(1,{\rm IV})$, the count is $\binom{n}{3}-n-n(n-4)$.  It counts the number of triples of edges $Z_iZ_{i(i+1)}$, and subtracts the ones in which at least two are connected. For types $(0,{\rm II})$,$(0,{\rm V})$,$(0,{\rm VI})$, the similar counts are, respectively, $n(n-5)/2$, $n(\binom{n-4}{2}-(n-5))$ and $\binom{n}{4}-n-n(n-5)-n(n-5)/2-n\binom{n-5}{2}$.
These are simplified in the table.  The rest of the counts are simpler and left to the reader. 

We conclude with a comment on Assumption \ref{assum:general}. A line $AB$ for which more than four of $\langle ABi(i+1)\rangle,i \in [n] $ vanish is contained in more than four of the boundary hyperplane sections of ${\rm Gr}_{\mathbb{C}}(2,4)$.  We may assume that it is a vertex on the the algebraic boundary.  Its description in terms of the above Schubert conditions would then not be unique, and the counting argument for the number of different strata would need to be modified.
\end{proof}

\begin{table}[h!]
\centering 
\footnotesize
\begin{tabular}{l|l|l|c|l}
name & Schubert conditions & indices & deg & b/r \\ 
\hline
$(3,{\rm I},i)$  & $L_{i(i+1)}$  & $i = 1, \ldots n$      &  1       &    b     \\
\hline
$(2,{\rm I}, i)$ & $V_i$ & $i = 1, \ldots, n$ & 1 &b \\
$(2,{\rm II}, i)$ & $P_{(i-1)i(i+1)}$ & $i = 1, \ldots, n$ & 1 &b \\
$(2,{\rm III}, ij)$ & $L_{i(i+1)}L_{j(j+1)}$ & $\{i,i+1\} \cap \{j, j+1\} = \emptyset$ & 2 &b \\
\hline
$(1, {\rm I},i)$ & $V_iL_{(i-1)(i+1)}$ & $i = 1, \ldots, n$ & 1 &b \\
$(1,{\rm II}, ij)$ & $V_i L_{j(j+1)}$ & $i \notin \{j,j+1\} $ & 1 &b \\
$(1,{\rm III}, ij)$ & $P_{(i-1)i(i+1)}L_{j(j+1)}$ & $\{i-1,i,i+1\} \cap \{j,j+1\} = \emptyset$ & 1 &r \\
$(1,{\rm IV},ijk)$ & $L_{i(i+1)} L_{j(j+1)} L_{k(k+1)}$ & $\{s,s+1\} \cap \{t,t+1\}  = \emptyset, \{s,t\} \in \binom{\{i,j,k\}}{2}$ & 2 &r \\ \hline
$(0,{\rm I},ij)$ & $V_iV_j$ & $\{i,j\} \in \binom{[n]}{2}$ & 1  &b \\
$(0,{\rm II},ij)$ & $P_{(i-1)i(i+1)}P_{(j-1)j(j+1)}$ & $\{i-1,i,i+1\} \cap \{j-1,j,j+1\} = \emptyset$ & 1 & r\\
$(0,{\rm III},ij)$ & $V_iL_{(i-1)(i+1)}L_{j(j+1)}$ & $\{i-1,i,i+1\} \cap \{j,j+1\} = \emptyset$ & 1 & r \\
$(0,{\rm IV},ijk)$ & $V_iL_{j(j+1)}L_{k(k+1)}$ & $\{j,j+1\} \cap \{k,k+1\} = \emptyset, i \notin\{j,j+1,k,k+1\}$ & 1 &r \\
$(0,{\rm V},ijk)$ & $P_{(i-1)i(i+1)}L_{j(j+1)}L_{k(k+1)}$ & 
$ {|s-i|>1, |s-i+1|>1, s \in \{j,k\},|j-k|>1} $ & 1 &r \\
$(0,{\rm VI},ijkl)$ & $L_{i(i+1)}L_{j(j+1)}L_{k(k+1)}L_{l(l+1)}$ & $\{s, s+ 1\} \cap \{t,t+1\} = \emptyset, \{s,t\} \in \binom{\{i,j,k,l\}}{2}$ & 2 &r
\end{tabular}
\caption{Closed strata of $\partial_a{\cal A}_n$}
\label{tab:complexstrata}
\end{table}

\begin{figure}
    \centering
    \includegraphics[width = \textwidth]{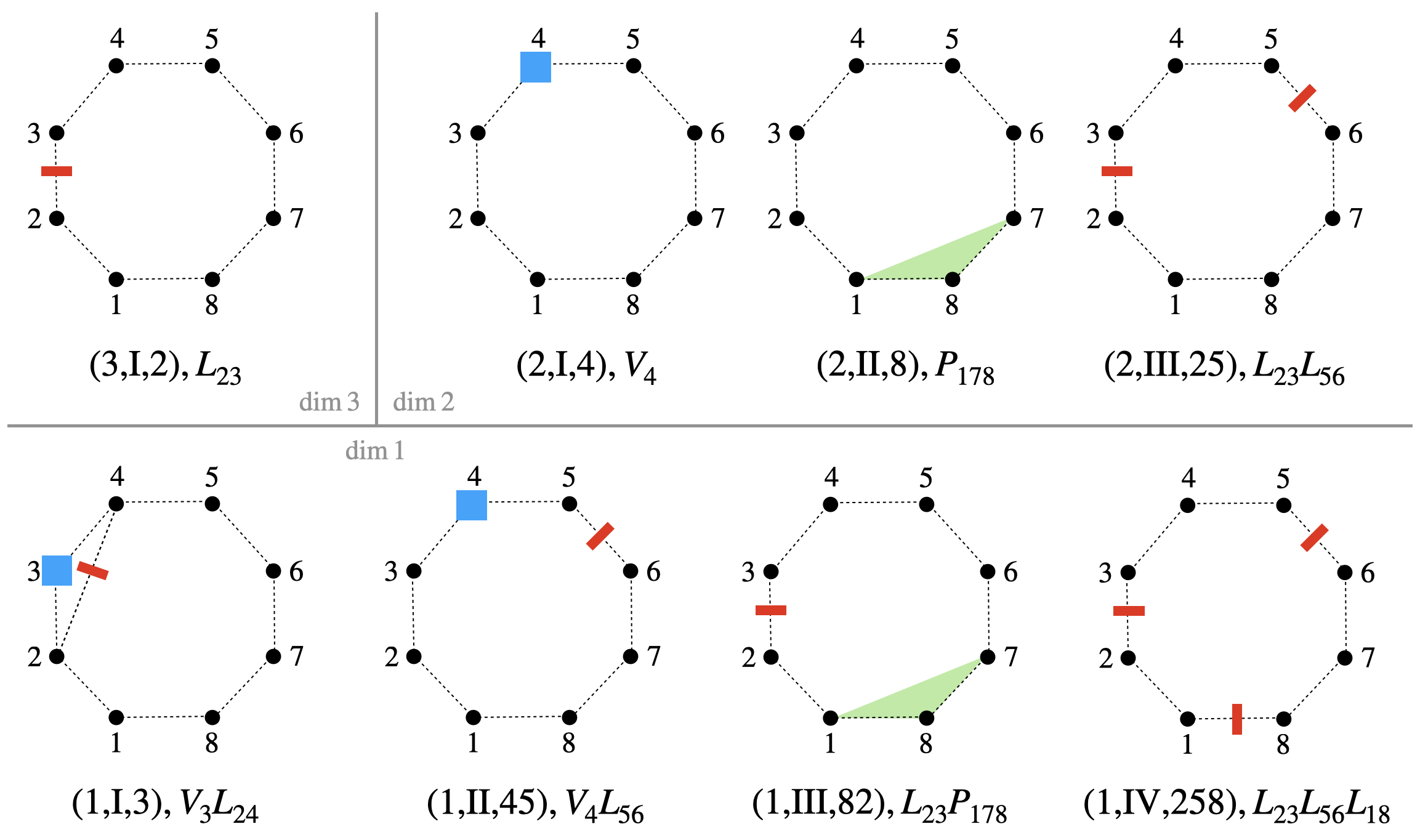}
    \includegraphics[width = \textwidth]{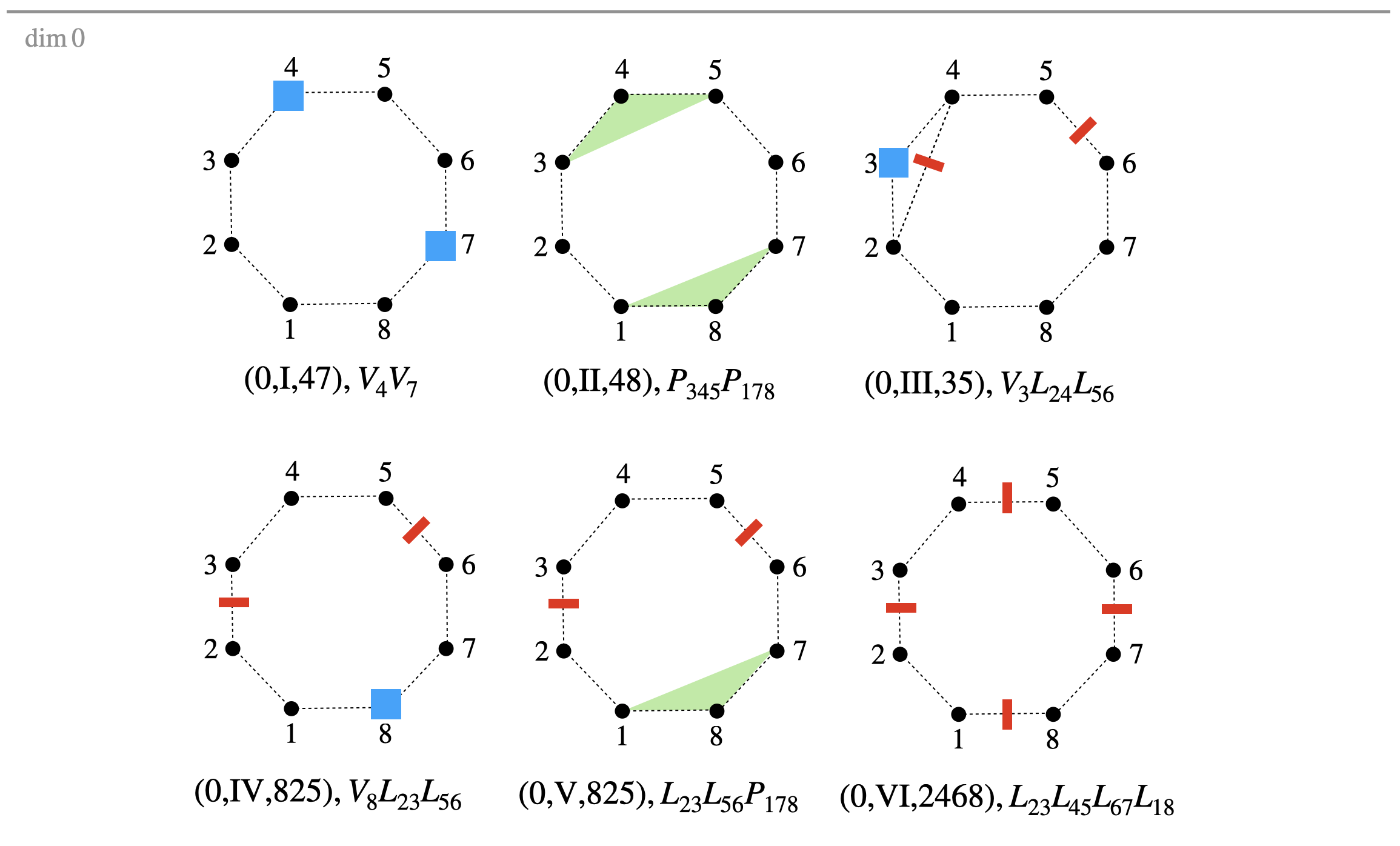}
    \caption{Closed strata of $\partial_a{\cal A}_8$}
    \label{fig:boundarystrata}
\end{figure}
 
We end the section with a discussion on the inclusion relations between closed strata. These are conveniently deduced from their description in terms of Schubert conditions. In many cases, an inclusion of closed strata is seen as the reverse inclusion of Schubert conditions. In particular, the relation $L_{(i-1)i}L_{i(i+1)} = V_i\cup P_{(i-1)i(i+1)}$ implies that the conditions $V_i$ and $P_{i-1,i,i+1}$ both imply $L_{(i-1)i}$ and $L_{i(i+1)}$. The inclusions between different types are coarsely depicted in Figure \ref{fig:rough stratification}. More precisely, an edge between two types P, Q of dimension $d, d+1$ respectively means that there is a closed stratum of type P contained in a closed stratum of type Q. 
\begin{figure}
\begin{equation}
\begin{tikzcd}
(3,{\rm I})  \arrow[d,,dash] \arrow[dr,dash] \arrow[drr,dash]& &&&&\\
(2,{\rm I}) \arrow[d,dash]\arrow[dr,dash] & (2,{\rm II})\arrow[d,dash] \arrow[dl,dash]\arrow[dr,dash]&(2,{\rm III}) \arrow[dl,dash] \arrow[d,dash]\arrow[dr,dash]&&&\\
(1,{\rm I}) \arrow[d,dash]\arrow[drr,dash] & (1,{\rm II}) \arrow[dl,dash]\arrow[dr,dash]\arrow[drr,dash]&(1,{\rm III})\arrow[dl,dash]\arrow[d,dash]\arrow[drr,dash]&(1,{\rm IV})\arrow[d,dash]\arrow[dr,dash]\arrow[drr,dash]&&\\
(0,{\rm I}) & (0,{\rm II}) &(0,{\rm III})&(0,{\rm IV}) & (0,{\rm V}) &(0,{\rm VI})
\end{tikzcd}
\end{equation}
\caption{Inclusion relations between types}
\label{fig:rough stratification}
\end{figure}
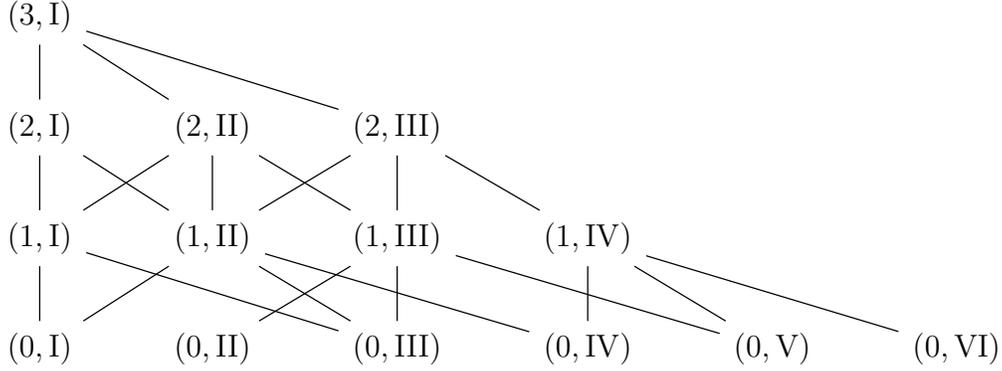
Table \ref{tab:vertices in curves} lists the inclusions of vertices in $\partial_a {\cal A}_n$  into $1$-dimensional strata.
\begin{table}[h!]
 \[
\begin{array}{ccc}
(0,{\rm I},i(i+1)) &\in & (1,{\rm I},i)\cap (1,{\rm I},i+1)\cap (1,{\rm II},i(i+1))\cap (1,{\rm II},(i+1)(i-1)) \\
(0,{\rm I},ij) &\in & (1,{\rm II},ij)\cap   (1,{\rm II},i(j-1))\cap (1,{\rm II},ji)\cap (1,{\rm II},j(i-1)) \\
(0,{\rm II},ij) &\in & (1,{\rm III},ij)\cap (1,{\rm III},i(j-1))\cap (1,{\rm III},ji)\cap (1,{\rm III},j(i-1))\\ 
(0,{\rm III},ij) &\in & (1,{\rm I},i)\cap (1,{\rm II},ij)\cap (1,{\rm III},ij)\\ 
(0,{\rm IV},ijk) &\in & (1,{\rm II},ij)\cap (1,{\rm II},ik)\cap (1,{\rm IV},(i-1)jk)\cap (1,{\rm IV},ijk)\\ 
(0,{\rm V},ijk) &\in & (1,{\rm III},ij)\cap (1,{\rm III},ik)\cap (1,{\rm IV},(i-1)jk)\cap (1,{\rm IV},ijk)\\
(0,{\rm VI},ijkl) &\in & (1,{\rm IV},ijk)\cap (1,{\rm IV},ijl)\cap (1,{\rm IV},ikl)\cap (1,{\rm IV},jkl)\\
\end{array}
\]
\caption{Incidence of $0$-dimensional strata of $\partial_a {\cal A}_n$ in $1$-dimensional strata}
\label{tab:vertices in curves}
\end{table}

\section{Face structure and residual arrangement} \label{sec:4}

The $n$ boundary facets of the amplituhedron are open subsets (in the Euclidean topology) of the $3$-dimensional strata. These strata are quadric boundary threefolds in the algebraic boundary. To describe boundary faces of lower dimension, we first investigate which strata in the algebraic boundary do not intersect the amplituhedron. The strata in $\partial_a {\cal A}_n$ are quadric threefolds, quadric surfaces, planes, conics, lines or just vertices (points). We will sometimes refer to them as such, instead of using the more general term ``strata''.

\begin{theorem} \label{thm:residualstrata}
 Under Assumption~\ref{assum:general}, the closed strata of $\partial_a {\cal A}_n(Z)$ which do not intersect ${\cal A}_n(Z)$ are  those contained in the strata of type $(1,{\rm III})$ or $(1,{\rm IV})$. These are all strata of type 
 \[ (1,{\rm III}), \,(1, {\rm IV}), \,(0, {\rm II}), \, (0, {\rm III}), (0, {\rm IV}), \, (0, {\rm V}), \, (0, {\rm VI}). \]
\end{theorem}

Throughout the section, we set ${\cal A}_n = {\cal A}_n(Z)$ for a fixed totally positive matrix $Z$, satisfying Assumption \ref{assum:general}.
We break up the proof of Theorem \ref{thm:residualstrata} into a series of lemmas.

\begin{lemma} \label{lem:codim2}
    All $2$-dimensional strata of $\partial_a {\cal A}_n$ intersect the amplituhedron ${\cal A}_n$ in a semi-algebraic set of dimension 2. 
\end{lemma}
\begin{proof}
    For each of the types $(2,{\rm I}), (2,{\rm II})$ and $(2,{\rm III})$, we describe a $2$-dimensional set of lines inside the intersection of a corresponding stratum and the amplituhedron~${\cal A}_n$. 

    We make use of the operator $\sigma$, introduced in the proof of Proposition \ref{prop:algboundary}.
    
    The $n$ substrata of type $(2,{\rm I})$ are planes in ${\rm Gr}(2,4) \subset \mathbb{P}^5$ consisting of lines through the point $Z_i$. For $i = 1, \ldots, n$, consider the following $2$-dimensional positroid cells of ${\rm Gr}(2,n)_{\geq 0}$:
    \[ P_i \, = \, \left \{ \sigma^{i-1} \begin{pmatrix}
        0 & 1 & x_1 & x_2 & 0 & \cdots & 0 \\ 
         -1 & 0 & 0 & 0 & 0 & \cdots & 0 
    \end{pmatrix}  \in \mathbb{R}^{2 \times n} \, : \, (x_1, x_2) \in \mathbb{R}^2_{>0} \right \}. \]
    The $2$-dimensional image of $P_i$ under $Z$ is contained in the plane $(2,{\rm I},i)$.
    
    We turn to type $(2,{\rm II})$. These are $n$ planes in ${\rm Gr}(2,4) \subset \mathbb{P}^5$ consisting of lines contained in the plane $Z_{i-1}Z_iZ_{i+1}$ for $i = 1, \ldots, n$, where $Z_{-1} = Z_n$ and $Z_{n+1} = Z_1$. We define 
    \[ Q_i \, = \, \left \{ \sigma^{i-1} \begin{pmatrix}
        x & 0 & 0 & \cdots & 0 & -1 \\
        y & 1 & 0 & \cdots & 0 & 0
    \end{pmatrix}\in \mathbb{R}^{2 \times n} \, : \, (x,y) \in \mathbb{R}^2_{>0} \right \}.
    \]
    The image $Z(Q_i)$ is a $2$-dimensional set of lines contained in the plane $(2,{\rm II},i)$. 

    Finally, there are $\binom{n}{2}-n$ strata of type $(2,{\rm III})$. These are quadric surfaces in ${\rm Gr}(2,4)$ consisting of lines which meet $Z_iZ_{i+1}$ and $Z_jZ_{j+1}$, where $\{i, i+1\} \cap \{j,j+1\} = \emptyset$. The image of the following positroid cell lies in the quadric surface $(2,{\rm III},ij)$:
    \setcounter{MaxMatrixCols}{20}
    \[ S_{i,j} \, = \, \left \{ \begin{pmatrix}
       0 & \cdots & 0 & 1 & x & 0 & \cdots & 0 & 0 & 0 & 0 & \cdots & 0 \\
       0 & \cdots & 0 & 0 & 0 & 0 & \cdots & 0 & y & 1 & 0 & \cdots & 0
    \end{pmatrix}  \in \mathbb{R}^{2,n} \, : \, (x,y) \in \mathbb{R}^2_{>0} \right \}.\]
    Here the first nonzero entry in the first row occurs in column $i$, and the first nonzero entry in the second row is in column $j$. If $j = n$, one uses the following matrix instead: 
    \[ \begin{pmatrix}
        0 & 0 &\cdots & 0 & 1 & x & 0 & \cdots & 0 & 0 \\
        -1 & 0 & \cdots & 0 & 0 & 0 & 0 & \cdots & 0 & y 
    \end{pmatrix}. \qedhere \]
\end{proof}

\begin{lemma}\label{lem:dotouch}
    All $1$-dimensional strata of $\partial_a {\cal A}_n$ of types $(1,{\rm I})$ and $(1,{\rm II})$ intersect the amplituhedron in a semi-algebraic set of dimension 1. 
\end{lemma}

\begin{proof} Let $\sigma$ be as in the proof of Proposition \ref{prop:algboundary}. 
The image of the positroid cell
    \[ \left \{  \sigma^{i-1} \begin{pmatrix}
        1 & 0 & 0 & \cdots & 0 & 0\\
        0 & x & 0 & \cdots & 0 & 1
    \end{pmatrix}  \in \mathbb{R}^{2 \times n} \, : \,  x \in \mathbb{R}_{>0}\right \}\]
under the amplituhedron map is a one-dimensional family of lines contained in the stratum $(1,{\rm I},i)$ of lines through $Z_i$ that meet the line $Z_{i-1}Z_{i+1}$.

For $i \notin \{j,j+1\}$, we consider the stratum $(1,{\rm II},ij)$ of lines through $Z_i$ that meet the line $Z_{j}Z_{j+1}$. If $j < n$, this stratum contains the image of the positroid cell
    \[ \left \{  \begin{pmatrix}
        0 & \cdots & 0 & 1 & 0 &  \cdots & 0 & 0 & 0 & 0 & \cdots & 0 \\
        0 & \cdots & 0 & 0 & 0 &  \cdots & 0 & 1 & y & 0 & \cdots & 0
    \end{pmatrix}  \in \mathbb{R}^{2 \times n} \, : \,  y \in \mathbb{R}_{>0}\right \}.\]
The first nonzero entries in the rows of the matrix are in positions $i$ and $j$ respectively. The case $j = n$ can be done similarly to the previous lemma.
\end{proof}

The next step is showing that other strata do not intersect ${\cal A}_n$. We use the following~fact.
\begin{lemma} \label{lem:useful}
    Let $AB \in {\cal A}_n$ be a line in the amplituhedron, and let $X \in {\rm Gr}(2,n)_{\geq 0}$ be such that $X \cdot Z = AB$. If $\langle ABi(i+1) \rangle = 0$ for some $i \in [n]$, then the $2 \times (n-2)$-submatrix $X_{[n] \setminus \{i,i+1\}}$ of $X$ consisting of columns not indexed by $i, i+1$ has rank one.
\end{lemma} 
\begin{proof}
    The bracket $\langle ABi(i+1) \rangle$ is the determinant of 
   \[ C_i \cdot Z \, = \, \begin{pmatrix}
       \times & \times & \cdots & \times & \times & \cdots & \times & \times \\
       \times & \times & \cdots & \times & \times & \cdots & \times & \times \\
       0 & 0 & \cdots & 1 & 0 & \cdots & 0 & 0 \\ 
       0 & 0 & \cdots & 0 & 1 & \cdots & 0 & 0
   \end{pmatrix} \cdot Z \, = \, \begin{pmatrix}
       X \cdot Z \\
       Z_i \\ 
       Z_{i+1}
   \end{pmatrix}, \]
   where the $\times$ symbols in $C_i$ represent entries of $X$, and the entries $1$ appear in columns $i$ and $i+1$. Since $AB$ hits $Z_iZ_{i+1}$, this determinant is zero. By the Cauchy-Binet formula, it also equals the sum of products of determinants $\langle ABi(i+1) \rangle = \sum_S |(C_i)_S| |Z_S|$, where $S \subset [n]$ with $ |S| = 4$ selects a $4 \times 4$ submatrix.  Since the identity block in the second row is consecutive, the nonnegativity of $X$ implies that the minors $|(C_i)_S|$ are nonnegative. With the positivity of $Z$ it follows that all terms in the Cauchy-Binet expansion are nonnegative, so that all maximal minors $|(C_i)_S|$ of $C_i$ are necessarily zero. These include all minors of $X_{[n] \setminus \{i, i+1\}}$, which concludes the proof. 
\end{proof}
\begin{lemma} \label{lem:3IIIbdoesnottouch}
    For $n \geq 5$, any $1$-dimensional closed stratum of $\partial_a{\cal A}_n$ of type $(1,{\rm III})$ does not intersect the amplituhedron ${\cal A}_n$. 
\end{lemma}
\begin{proof}
A $1$-dimensional stratum of $\partial_a {\cal A}_n$ of type $(1,{\rm III})$ is determined by two indices $i,j$ satisfying $\{i-1,i,i+1\} \cap \{j,j+1\} = \emptyset$. It consists of the lines $AB$ contained in the plane $Z_{i-1}Z_{i}Z_{i+1}$ and intersecting the line $Z_{j}Z_{j+1}$.
   Suppose $AB = X \cdot Z$ lies in the amplituhedron, and $\langle AB(i-1)i\rangle = \langle ABi(i+1) \rangle = \langle AB j(j+1) \rangle = 0$. Then the submatrices $X_{[n] \setminus \{i-1,i\}}$, $X_{[n] \setminus \{i,i+1\}}$ and $X_{[n] \setminus \{j,j+1\}}$ have rank 1 by Lemma \ref{lem:useful}. Every column belongs to at least two of these matrices, except column $i$. Since ${\rm rank}(X)=2$, we can bring $X$ into the~form 
   \[ X \, = \, \begin{pmatrix}
       0 & \cdots & 0 & \times & 0 & \cdots & 0 \\
       \times & \cdots & \times & \times & \times & \cdots & \times 
   \end{pmatrix}.\]
   Here the only nonzero entry in the first row appears in the $i$-th column. We conclude that $AB = X \cdot Z$ is a line through $Z_i$, so that it is contained in the closed stratum $(1, {\rm II},ij)$ given by $V_iL_{j(j+1)}$. The unique intersection point of the
   $(1,{\rm II},ij)$-stratum with $(1,{\rm III},ij)$ is the line through $Z_i$ and through the intersection point of the line $Z_jZ_{j+1}$ with the plane $Z_{i-1}Z_iZ_{i+1}$. It remains to show that this line, denoted $AB$ in what follows, is not contained in the amplituhedron. Let $\langle ijkl\rangle$ denote the $4 \times 4$-minor of $Z$ indexed by $i,j,k,l$. We can~set
   \[A \, = \, Z_i, \quad B \, = \, \langle i (i+1)k(k+1)\rangle  Z_{i-1} - \langle (i-1)(i+1)k(k+1) \rangle Z_i+ \langle (i-1)ik(k+1)\rangle Z_{i+1}.\]
   The expression for $B$ is the intersection point mentioned above. We use this to compute
   \begin{align*}
       \langle AB(i-2)(i-1) \rangle \, &= \,  \langle (i-1)ik(k+1) \rangle \cdot \langle i(i+1)(i-2)(i-1) \rangle > 0,\\
       \langle AB(i+1)(i+2) \rangle \, &= \, \langle i(i+1)k(k+1) \rangle \cdot \langle i(i-1)(i+1)(i+2) \rangle < 0.
   \end{align*}
  The inequalities follow from the positivity of $Z$. But ${\rm sign}(\langle AB(i-2)(i-1) \rangle) = - {\rm sign}(\langle AB(i+1)(i+2) \rangle)$ violates the inequalities from Definition \ref{def:amplituhedron2}, so $AB \notin {\cal A}_n$. 
\end{proof}

\begin{lemma} \label{lem:3IIdoesnottouch}
For $n \geq 6$, any $1$-dimensional closed stratum of $\partial_a {\cal A}_n$ of type $(1,{\rm IV})$ does not intersect the amplituhedron ${\cal A}_n$.
\end{lemma}
\begin{proof}
   A $1$-dimensional stratum of $\partial_a {\cal A}_n$ of type $(1,{\rm IV})$ is determined by three indices $i,j,k$ satisfying $\{s,s+1\} \cap \{t,t+1\}  = \emptyset, \{s,t\} \in \binom{\{i,j,k\}}{2}$. It consists of the lines $AB$ passing through the lines $Z_{i}Z_{i+1}, Z_{j}Z_{j+1}$ and $Z_{k}Z_{k+1}$. If such a line is in ${\cal A}_n$, it is given by $X \cdot Z \in {\rm Gr}(2,4)$, where $X \in {\rm Gr}(2,n)_{\geq 0}$ is represented by a rank-$2$ totally nonnegative $2 \times n$-matrix. The submatrices $X_{[n] \setminus \{i, i+1\}}$, $X_{[n] \setminus \{j, j+1\}}$ and $X_{[n] \setminus \{k, k+1\}}$ have rank one by Lemma \ref{lem:useful}. Each column belongs to at least two such submatrices. Hence $X$ has rank one, a contradiction. 
\end{proof}

\begin{proof}[Proof of Theorem \ref{thm:residualstrata}]
Note that the vertices in the algebraic boundary of type $(0,{\rm I})$ lie on the amplituhedron. By the inclusion description in Table \ref{tab:vertices in curves}, all other vertices are contained in $1$-dimensional strata of type $(1,{\rm III})$ or $(1,{\rm IV})$, which by Lemmas \ref{lem:3IIdoesnottouch} and \ref{lem:3IIIbdoesnottouch} do not intersect the amplituhedron.
With Lemmas \ref{lem:codim2} and \ref{lem:dotouch} the proof is complete.
\end{proof}

The last column of Table \ref{tab:complexstrata} summarizes which closed strata intersect the amplituhedron (b) and which do not (r). The letters b and r stand for \emph{boundary} and \emph{residual}.
The strata which do not intersect the amplituhedron, i.e., those identified in Theorem \ref{thm:residualstrata} form the \emph{residual arrangement} of the amplituhedron, see Example \ref{ex:canformpentagon} and Definition \ref{def:resarr}. The others contain the boundary faces of the amplituhedron. 
We use the word face here in analogy to convex sets (like polytopes), but mean the relative interior of the positive geometry induced on the irreducible component. The boundary of each face is a union of faces of smaller dimension. With the description of incidences of strata in $\partial_a {\cal A}_n$ from Section \ref{sec:3}, and the characterization of the strata that do not intersect the amplituhedron (Theorem \ref{thm:residualstrata}), we can describe the boundary face structure of the amplituhedron completely. 

Our next Proposition gives an exhaustive list of all faces. It essentially summarizes the stratification in Section \ref{sec:3} and the conclusion of Theorem \ref{thm:residualstrata}. Figures \ref{fig:schlegel3D} and \ref{fig:schlegel4D} help to verify this list. Figure \ref{fig:schlegel3D} shows a Schlegel diagram of the 3-dimensional boundary polytope $(3,{\rm I},1)\cap {\cal A}_7$ of ${\cal A}_7$, with respect to its facet $(2,{\rm I},1) \cap {\cal A}_7$. Boundary line segments are labeled in blue, and boundary polygons in red. Figure \ref{fig:schlegel4D} shows the Schlegel diagram of ${\cal A}_5$ and ${\cal A}_6$, both with respect to their 3-dimensional boundary polytope inside $L_{12}$. The edges of that polytope are shown in blue. The black edges are in the interior. The vertex $V_iV_j$ is labeled~$ij$. 

\begin{proposition}\label{prop:boundary faces} The boundary faces, by dimension, of the amplituhedron are the~following.
\begin{itemize}
    \item Boundary vertices are the $0$-dimensional strata of type $(0,{\rm I})$.
    The vertices of type $(0,{\rm I},i(i+1))$ lie on the four boundary lines  $(1,{\rm I},i)$, $(1,{\rm I},i+1)$, $(1,{\rm II},i(i+1))$ and $(1,{\rm II},(i+1)(i-1))$.
    The vertices of type $(0,{\rm I},ij),j \notin \{i-1,i,i+1\}$ lie on the four boundary lines  $(1,{\rm II},i(j-1))$, $(1,{\rm II},ij)$, $(1,{\rm II},j(i-1))$ and $(1,{\rm II},ji)$.
    \item Boundary line segments are contained in the lines of type $(1,{\rm I})$ and $(1,{\rm II})$.  The line segment in $(1,{\rm I},i)$ has the two boundary vertices $(0,{\rm I},(i-1)i)$ and $(0,{\rm I},i(i+1))$. The line segment in $(1,{\rm II},ij)$ has the two boundary vertices $(0,{\rm I},ij)$ and $(0,{\rm I},i(j+1))$. 
    \item Boundary polygons are contained in the planes of type $(2,{\rm I})$ and $(2,{\rm II})$ and boundary quadrilaterals are contained in the quadric surfaces of type $(2,{\rm III})$.
    The plane $(2,{\rm I},i)$ contains the convex boundary $(n-1)$-gon with vertices in cyclic order $(0,{\rm I},ij),j=i+1,\ldots,i-1$. Its $n-1$ edges are contained in 
    the line $(1,{\rm I},i)$ and the $n-2$ lines $(1,{\rm II},ij), j\notin\{i-1,i\}$.  The plane $(2,{\rm II},i)$ contains the boundary triangle with vertices 
    $(0,{\rm I},(i-1)i),(0,{\rm I},i(i+1)),(0,{\rm I},(i-1)(i+1))$. Its edges are contained in the lines $(1,{\rm II},(i-1)i)$,$(1,{\rm I},i)$ and $(1,{\rm II},(i+1)(i-1))$. 
    The quadric surface $(2,{\rm III},ij)$ contains the boundary quadrilateral with vertices, in cyclic order,  $(0,{\rm I},ij)$, $(0,{\rm I},(i+1)j)$, $(0,{\rm I},(i+1)(j+1))$ and $(0,{\rm I},i(j+1))$. The four edges are contained in the 
    lines $(1,{\rm II},ji)$,$(1,{\rm II},(i+1)j)$,$(1,{\rm II},(j+1)i)$,$(1,{\rm II},ij)$.
    \item Boundary $3$-dimensional polytopes are contained in the quadric threefolds of type $(3,{\rm I})$. The boundary polytope inside the quadric threefold $(3,{\rm I},i)$ has $n+1$ facets. Among those, $n-3$ are the boundary quadrilaterals in the quadric surfaces $(2,{\rm III},ij), j\notin\{i-1,i,i+1\}$, two are boundary $(n-1)$-gons in the planes $(2,{\rm I},i)$ $(2,{\rm I},i+1)$  and two are boundary triangles in the planes $(2,{\rm II},i)$ $(2,{\rm II},i+1)$. These planes and quadric surfaces form $n+1$ hyperplane sections of the quadric threefold $(3,{\rm I},i)$. 
\end{itemize}
\end{proposition}

\begin{proof}
It remains to note that the $(n-1)$-gon in the plane of type $(2,{\rm I})$ is a convex $(n-1)$-gon, by the assumption that $Z$ is positive. 
\end{proof}

\begin{figure}[h!]
    \centering
    \includegraphics[height = 6cm]{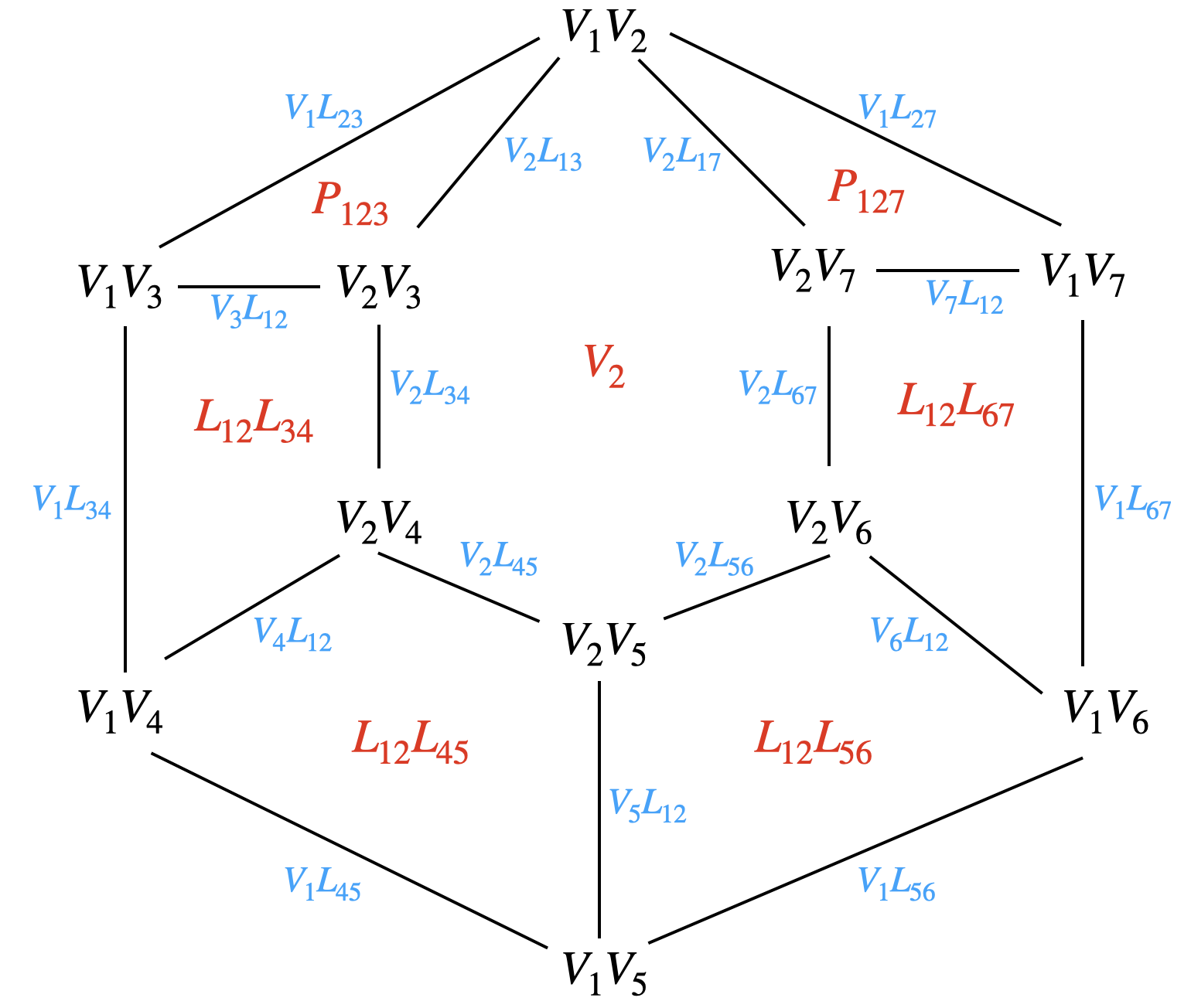}
    \caption{Schlegel diagram of the 3-dimensional boundary $ (3,{\rm I},1) \cap {\cal A}_7$.}
    \label{fig:schlegel3D}
\end{figure}

\begin{figure}[h!]
    \centering
    \includegraphics[height = 5cm]{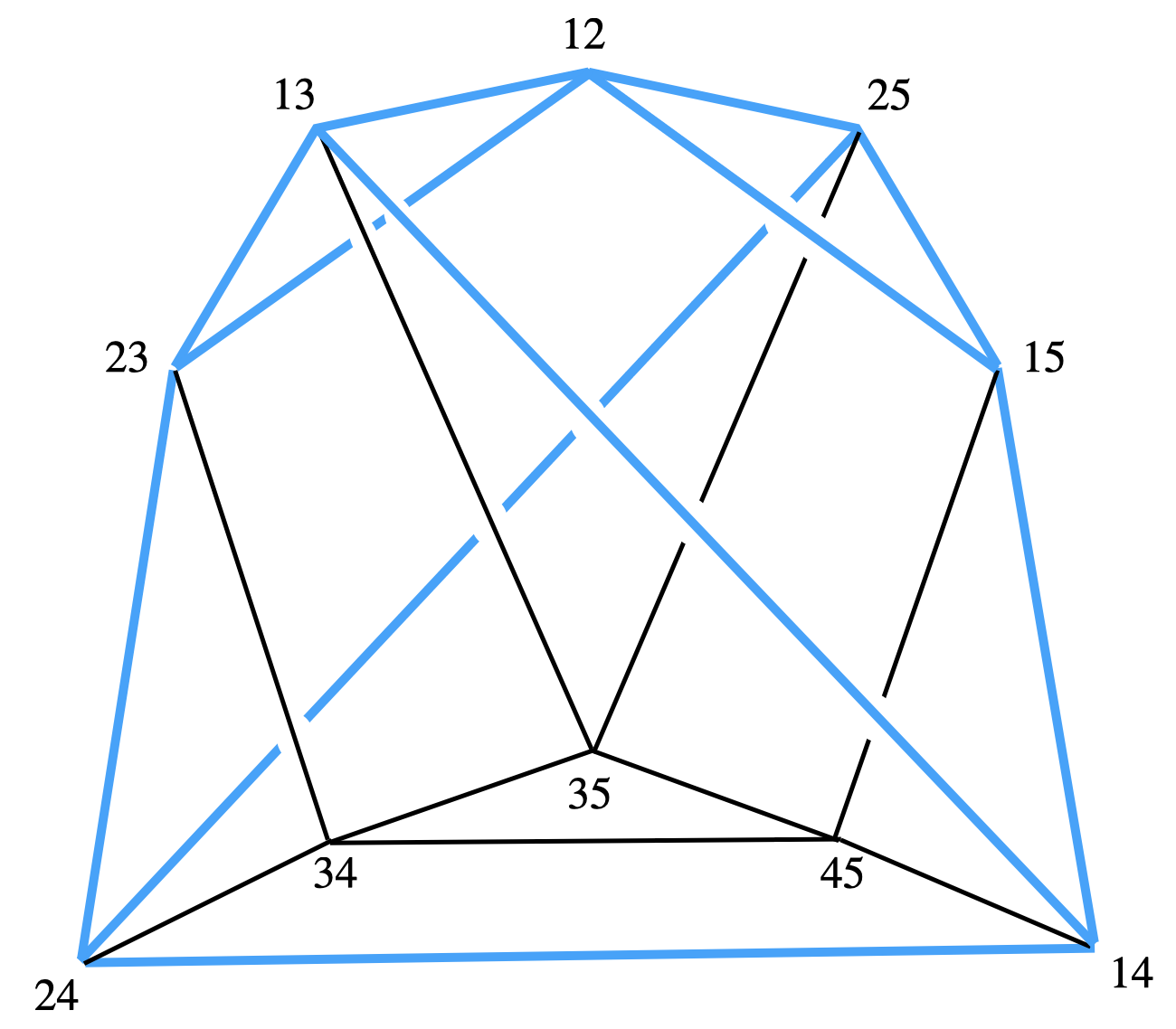} \quad 
    \includegraphics[height = 5cm]{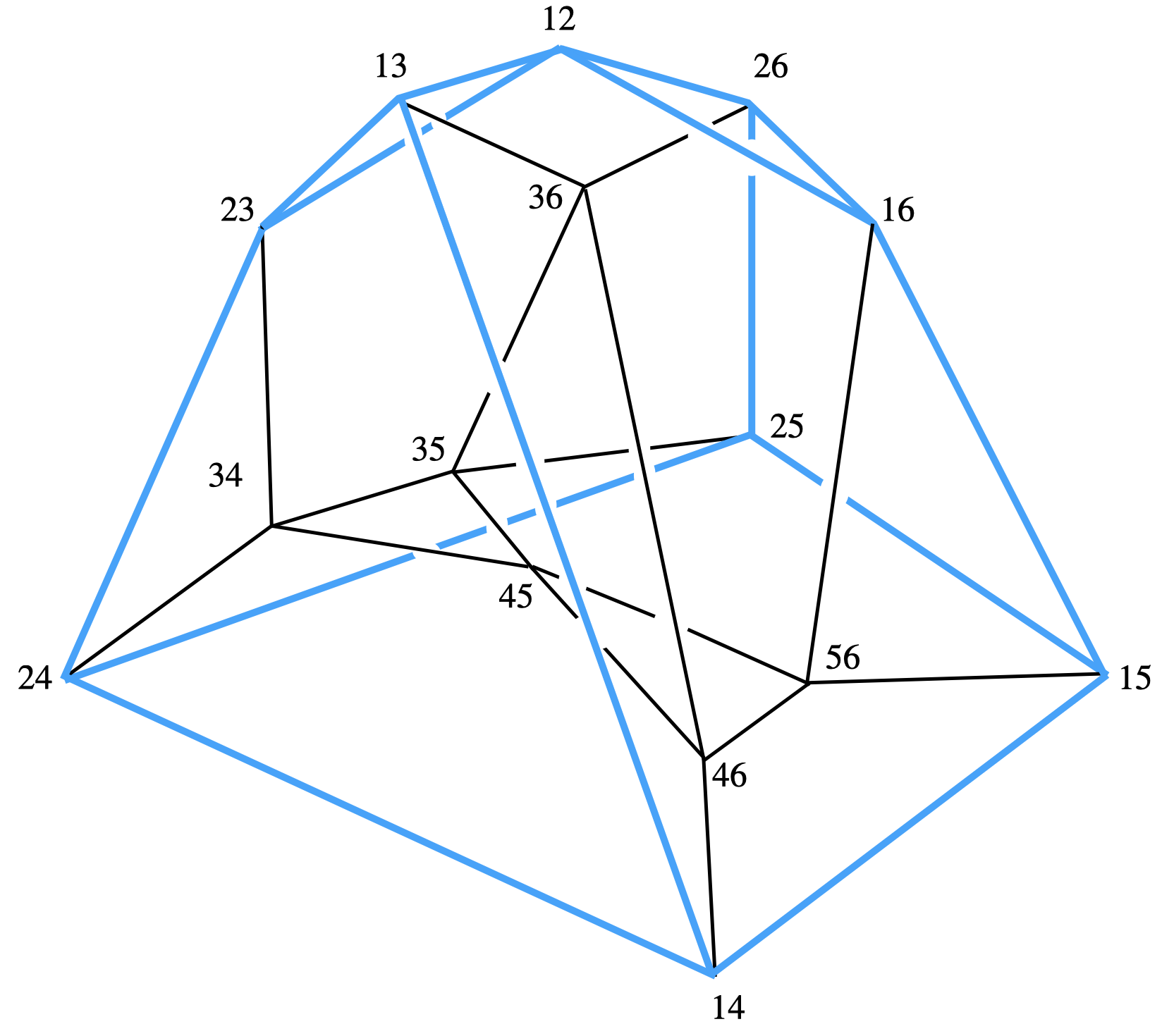}
    \caption{Schlegel diagrams of ${\cal A}_5$ and ${\cal A}_6$. }
    \label{fig:schlegel4D}
\end{figure}

\begin{corollary}
Every boundary vertex of the amplituhedron is contained in at most four boundary facets. Each $3$-dimensional facet is combinatorially a polytope in a quadric threefold of type $(3,{\rm I})$. Its~only non-simple vertex is $(0,{\rm I},i(i+1))$, where two plane triangles and two plane $(n-1)$-gons meet. 
    \end{corollary}
\begin{proof}
 The vertex $(0,{\rm I},i(i+1))$ lies in the boundary facets of the three threefolds $(3,{\rm I},i-1)$, $(3,{\rm I},i)$ and $(3,{\rm I},i+1)$, while the vertices $(0,{\rm I},ij), j\notin\{i-1,i+1\}$ are contained in the boundary facets of the four threefolds $(3,{\rm I},i-1)$, $(3,{\rm I},i)$, $(3,{\rm I},j-1)$ and $(3,{\rm I},j)$.  Similarly, the boundary line segments are contained in at most three boundary polygons, and the boundary polygons are contained in at most two facets.

The quadric threefold $(3,{\rm I},i)$ is a cone with vertex $(0,{\rm I},i(i+1))$. The boundary facet in $(3,{\rm I},i)$ contains boundary polygons in the four planes $(2,{\rm I},i), (2,{\rm I},i+1),(2,{\rm II},i), (2,{\rm II},i+1)$, all passing through that vertex. The boundary quadrilaterals in the $n-3$ smooth quadric surfaces $(2,{\rm III},ij), j\notin \{i-1,i,i+1\}$ do not pass through $(0,{\rm I},i(i+1))$. Beside $(0,{\rm I},i(i+1))$, there are the $2(n-2)$ boundary vertices $(0,{\rm I},jk); j\in \{i,i+1\}, k\notin\{i,i+1\}$ on the $(3,{\rm I},i)$ facet. Each of these is the intersection of three boundary polygons, see Table \ref{tab:vertsoffacet}.
\end{proof}

\begin{table}[h!]
\centering 
\begin{tabular}{l|l}
vertex of $(3,{\rm I},i) \cap {\cal A}_n$ & is contained in the polygons of \\ \hline
$(0,{\rm I},ik), \, k \notin \{i-1,i,i+1,i+2\}$ & $(2,{\rm I},i), \, (2,{\rm III},ik), \, (2,{\rm III},i(k-1)) $ \\
$(0,{\rm I},(i+1)k), \, k \notin \{i-1,i,i+1,i+2\}$ & $(2,{\rm I},i+1), \, (2,{\rm III},ik), \, (2,{\rm III},i(k-1)) $ \\
$(0,{\rm I}, (i-1)i)$ & $(2,{\rm I},i), \, (2,{\rm II},i), \, (2,{\rm III},i(i-2)) $ \\
$(0,{\rm I}, i(i+2))$ & $(2,{\rm I},i), \, (2,{\rm II},i+1), \, (2,{\rm III},i(i+2)) $ \\
$(0,{\rm I}, (i-1)(i+1))$ & $(2,{\rm I},i+1), \, (2,{\rm II},i), \, (2,{\rm III},i(i-2)) $ \\
$(0,{\rm I}, (i+1)(i+2))$ & $(2,{\rm I},i+1), \, (2,{\rm II},i+1), \, (2,{\rm III},i(i+2)) $ \\
\end{tabular}
\caption{Vertex-facet incidence of the 3-dimensional boundary $(3,{\rm I},i) \cap {\cal A}_n$.}
\label{tab:vertsoffacet}
\end{table}

\begin{corollary}\label{cor:connected one-skeleton} The one-dimensional skeleton of the boundary, i.e., the union of $1$-dimensional faces, is connected.  
    \end{corollary}
    \begin{proof}The boundary line segments in lines of type $(1,{\rm I},i)$ form a cyclic $n$-gon with vertices at the points $(0,{\rm I},i(i+1))$. The boundary line segment in the line $(1,{\rm II},ij)$ connects the boundary vertices  $(0,{\rm I},ij)$ and $(0,{\rm I},i(j+1))$.  When $i$ and $j$ varies any two boundary vertices are connected via a finite number of boundary line segments. 
    \end{proof}

    Having desribed the face structure of the boundary of the amplituhedron, we now turn to the strata in the algebraic boundary that do not intersect the amplituhedron.
    These form the residual arrangement ${\cal R}({\cal A}_n)$ of our amplituhedron ${\cal A}_n$. 
\begin{definition} \label{def:resarr}
The \emph{residual arrangement} ${\cal R}({\cal A}_n)$ of the amplituhedron ${\cal A}_n$ is the union of all strata in the algebraic boundary that do not intersect the amplituhedron.
\end{definition}
Theorem \ref{thm:residualstrata} says that the residual arrangement is the union of the open strata 
$$(1,{\rm III}), (1,{\rm IV}),(0,{\rm II}), (0,{\rm III}),(0,{\rm IV}), (0,{\rm V}), (0,{\rm VI}).$$
This equals the union of the closed strata of type $(1,{\rm III})$ and $(1, {\rm IV})$.
The strata of type $(1,{\rm III})$ are the {\em residual lines} in ${\cal R}({\cal A}_n)$, while
the strata of type $(1,{\rm IV})$ are the {\em residual conics} in ${\cal R}(\cA_n)$.  We call the $0$-dimensional strata in ${\cal R}(\cA_n)$ the  {\em residual vertices} of ${\cal R}(\cA_n)$.

\section{Unique adjoint threefold} \label{sec:5}
In this section, we discuss the analogon to the adjoint hypersurface of a simple polytope. In the context of Wachspress coordinates and polypols, similar generalizations have appeared in the literature, see \cite{kohn2021adjoints,kohn2020projective,lam2022invitation}.
Consider a simple convex polytope $P \subset \PP^d$ of dimension $d$ with $n$ facets in hyperplanes $D_1, \ldots, D_n$. In Wachspress geometry, one defines a hypersurface given by $\alpha_P(x) = 0$, called the \emph{adjoint hypersurface}, by requiring it to be of degree $n-d-1$, and to interpolate a \emph{residual arrangement} ${\cal R}(P)$ \cite{kohn2020projective}. The latter is the union of all linear spaces that are intersections of a subset of the hyperplanes $D_1, \ldots, D_n$, but do not contain any face of $P$. 
An example of a polygon is shown in Figure \ref{fig:adjpentagon}. The adjoint curve of a pentagon is the unique conic passing through the five blue points in its residual arrangement. 
In \cite{kohn2021adjoints}, this was generalized from polytopes to \emph{rational polypols}. Among those are rational polypols representing semi-algebraic subsets of the plane with rational boundary curves. Definition \ref{def:resarr} of the residual arrangement is in direct analogy with these examples. 
Our interest in adjoint curves of polypols stems from the fact that their defining equation is the numerator of the 
\emph{canonical form} of the polypol as a positive geometry \cite[Theorem 2.15]{kohn2021adjoints}. 
We generalize this to amplituhedra, and use the adjoint to show that these are positive geometries in~Section~\ref{sec:6}.

Let $R = \C[p_{12},p_{13},p_{14},p_{23},p_{24},p_{34}]/ \langle p_{12}p_{34}-p_{13}p_{24}+p_{14}p_{23} \rangle$ be the homogeneous coordinate ring of ${\rm Gr}_\C(2,4)$. Its degree $d$ part is denoted  by $R_d$. 
\begin{definition}
A form $\alpha_{{\cal A}_n} \in R_{n-4}$ is an \emph{adjoint polynomial} for ${\cal A}_n$ if $\alpha_{\A_n}(Y) = 0$ for all $Y \in {\cal R}(\cA_n)$. The zero locus $\{Y \in {\rm Gr}_\C(2,4) ~|~ \alpha_{{\cal A}_n}(Y) = 0 \}$ is called an \emph{adjoint three-fold}.
\end{definition}

In this section, we prove the following statement.
\begin{theorem}\label{adjoint}
    Under Assumption~\ref{assum:general}, there exists a unique, up to scaling, adjoint polynomial $\alpha_{\cA_n}$ for $\cA_n$ in the homogeneous coordinate ring $R$ of the Grassmannian ${\rm Gr}_{\mathbb{C}}(2,4)$.
\end{theorem}

Below, we set ${\cal A}_n = {\cal A}_n(Z)$ for a fixed totally positive matrix $Z$ satisfying Assumption \ref{assum:general}. To determine $\alpha_{{\cal A}_n}$ uniquely, we need $\dim_\C R_{n-4} - 1$ independent conditions. We have 
\begin{equation} \label{eq:formuladimR}
    \dim_\C R_{n-4} \, = \, \binom{n+1}{5} - \binom{n-1}{5}. 
\end{equation} 
We will obtain such conditions as interpolation conditions for distinguished lines in the residual arrangement ${\cal R}({\cal A}_n)$. Namely, the adjoint threefold is required to interpolate the lines of type $(1,{\rm III})$ and $(1,{\rm IV})$, see Theorem \ref{thm:residualstrata}. 
First, we work out the cases $n = 4, 5$.

\begin{example}
For $n = 4$, the algebraic boundary $\partial_a{\cal A}_n$ contains 
 no strata of types $(1,{\rm III})$ and $(1,{\rm IV})$ and only vertices of type $(0,{\rm I})$ (Theorem \ref{thm:stratanew}),  hence ${\cal R}({\cal A}_4) = \emptyset$ and $\alpha_{{\cal A}_4} = 1$.
\end{example} 

\begin{example}
The algebraic boundary of ${\cal A}_5(Z)$ contains five strata of type $(1,{\rm III})$ and no strata of type $(1,{\rm IV})$, so the residual arrangement ${\cal R}_5(Z)$ consists of the five residual lines of type $(1,{\rm III})$.
They form a cycle with vertices at the five residual vertices of type $(0,{\rm IV})$. This cycle also contains five residual vertices of type $(0,{\rm III})$ one on each residual line.
The cycle uniquely determines a linear form in $R_1$, which is the adjoint polynomial $\alpha_{{\cal A}_5}$ of ${\cal A}_5(Z)$. 

More explicitly, the adjoint $\alpha_{{\cal A}_5}$ interpolates the five residual vertices 
\[
(0,{\rm IV},124),(0,{\rm IV},235),(0,{\rm IV},341),(0,{\rm IV},452),(0,{\rm IV},513).
\]
Since each residual line contains two of these vertices, the resulting linear form will automatically vanish on ${\cal R}_5(Z)$. The Pl\"ucker coordinates of the vertex $(0,{\rm IV},124)$ represent the corresponding interpolation condition for the linear form $\alpha_{{\cal A}_5}$. To compute these Pl\"ucker coordinates, note that they satisfy the linear relations  $\langle AB12 \rangle \, =\langle AB13 \rangle \, =\langle AB14 \rangle \, =\langle AB15 \rangle \, =\langle AB23 \rangle \, =\langle AB45 \rangle \, =0$.  The first four forms are linearly dependent, but any three of them and the last two are independent, so these linear forms have a unique solution. The interpolation conditions for the other four vertices are obtained in a similar fashion. 

We give a useful recipe for computing $\alpha_{{\cal A}_5}$. For $\{i,j\} \subset [5]$, let $c_{ij} = \prod_{k \in [5] \setminus \{i,j\}} \langle Z_{\hat{k}} \rangle$, where $\langle Z_{\hat{k}} \rangle$ denotes the $4 \times 4$ minor of $Z$ obtained by deleting the $k$-th row.
The adjoint polynomial $\alpha_{{\cal A}_5}$ is the linear form $d_{12} p_{12} + d_{13} p_{13} + d_{14} p_{14} + d_{23} p_{23} + d_{24} p_{24} + d_{34} p_{34}$ whose coefficients $d$ are such that $\big ( \bigwedge^2 Z \big )  \cdot d = c$. Here $\bigwedge^2 Z$ is the $10 \times 6$ matrix whose $(ij,kl)$ entry is the $2 \times 2$ minor of $Z$ with rows $i,j$ and columns $k,l$, and the $ij$-entry of $c$ is $c_{ij}$.
\end{example}

\begin{proposition} \label{prop:countres}
The number of residual vertices in ${\cal R}({\cal A}_n)$ is given by 
\begin{equation} \label{eq:countres}   \frac{1}{12}n^{4}-\frac{1}{2}n^{3}+\frac{17}{12}n^{2}-3n . 
\end{equation}
\end{proposition}
\begin{proof}
By Theorem \ref{thm:residualstrata}, the correct number is obtained as 
\[ \#(0, {\rm II})+ \# (0,{\rm III})+ \# (0, {\rm IV})+ 
\# (0, {\rm V})+ 2\cdot \# (0, {\rm VI})
, \]
where $\#(\cdot)$ denotes the number of substrata of a given type. Note that the 2 in front of $\# (0,{\rm VI})$ stems from the fact that each stratum of type $(0,{\rm VI})$ consists of two points outside ${\cal A}_n$. 
Plugging in the expressions from Theorem \ref{thm:stratanew} gives \eqref{eq:countres}.
\end{proof}

Using Proposition \ref{prop:countres} and \eqref{eq:formuladimR}, we get for all integers $n\geq 4$ 
\[ \frac{1}{12}n^{4}-\frac{1}{2}n^{3}+\frac{17}{12}n^{2}-3n - (\dim_\C R_{n-4} - 1) \, = \,  \frac{1}{6} n (n+1) (n-4) \geq 0.\]
This also shows that, for $n \geq 5$, we have more interpolation conditions than degrees of freedom.  Notice that the number of seemingly superfluous interpolation conditions coincides with the number of dimension 1 substrata in ${\cal R}({\cal A}_n)$ of type $(1,{\rm III})$ and $(1,{\rm IV})$.  In fact,
\[ \#(1, {\rm III})+ \# (1,{\rm IV})=n(n-4)+\binom{n}{3}-n(n-3)=\frac{1}{6} n (n+1) (n-4).
\]

The following proposition is originally proven by Arkani-Hamed, Hodges and Trnka.
\begin{proposition} \label{prop:existence}\cite[Section 2.2] {arkani2015positive}
Each  residual line contains $n-2$ residual vertices, and each residual conic contains $2n-6$ residual vertices.
In particular, ${\cal R}({\cal A}_n)$ imposes at most 
$(\dim_\C R_{n-4} - 1) $ linearly independent interpolation conditions on forms in $R_{n-4}$, so there is at least one nonzero polynomial $f\in R_{n-4}$ that vanishes on ${\cal R}({\cal A}_n)$.
\end{proposition}
\begin{proof}

Each residual line and conic in ${\cal R}({\cal A}_n)$ comes with marked lines $Z_iZ_{i+1}$, in either red or green in Figure~\ref{fig:boundarystrata}. The lines of type $(1,{\rm III})$ contain one residual vertex for each of the $n-3$ unmarked lines $Z_jZ_{j+1}$ and in addition one vertex of type $(0,{\rm III})$. Similarly, each conic of type $(1, {\rm IV})$ contain two vertices per unmarked line. We describe these vertices for each of the above kinds in Remark \ref{vertextypes}.
 
This proves the first part of the proposition. For the second part, notice that $n-3$ points on a line and $2n-7$ points on a conic impose independent conditions on forms of degree $n-4$. So at least one condition on each line and one condition on each conic is superfluous.
\end{proof}

\begin{remark}\label{vertextypes}
We may distinguish different kinds of lines/conics in ${\cal R}({\cal A}_n)$ to describe more precisely which boundary vertices lie on each curve. We say that the line $(1,{\rm III},ij)$ is 
\begin{itemize}
    \item of the \emph{first kind} if the distance in the cyclic order between the plane $Z_{i-1}Z_iZ_{i+1}$ and the marked line $Z_jZ_{j+1}$ is one,
    \item and of the \emph{second kind} if this distance is greater than one.
\end{itemize}  
We say that a $(1,{\rm IV},ijk)$-conic  
\begin{itemize}
    \item is of the \emph{first kind} if one of the marked lines has distance one to the two others, 
    \item it is of the \emph{second kind} if two marked lines have distance one, while the third has distance more than one to the first two, 
    \item and it is of the \emph{third kind} if no two of the three have distance one.
\end{itemize}
A residual line of the first kind, say $(1,{\rm III},i(i+2))$,   contains the vertices $(0,{\rm III},i(i+2))$, $(0,{\rm IV},(i-1)i(i+2))$, $(0,{\rm IV},(i+1)(i-1)(i+2))$ and $(0,{\rm II},i(i+3))$  and $n-6$ vertices of type $(0,{\rm V})$.  A residual line $(1,{\rm III},ij)$ of the second kind contains the vertices $(0,{\rm III,ij})$, $(0,{\rm II},ij)$, $(0,{\rm II},i(j+1))$,$(0,{\rm IV},(i-1)ij)$, $(0,{\rm IV},(i+1)(i-1)(j+1))$ and $n-7$ of type $(0,{\rm V})$.

A type-$(3, \rm II)$ conic of the first kind
 contains six vertices of type $(0,{\rm IV})$,
  two of type $(0,{\rm V})$, and $2(n-7)$ of type $(0,{\rm VI})$.  A residual conic of the second kind 
 contains six vertices of type $(0,{\rm IV})$, four of type $(0,{\rm V})$, and $2(n-8)$ of type $(0,{\rm VI})$. A conic of the third kind 
 contains 
 six vertices of type $(0,{\rm IV})$, six of type $(0,{\rm V})$, and $2(n-9)$ of type $(0,{\rm VI})$.
\end{remark}

To show that there is at most one nonzero polynomial $f\in R_{n-4}$ that vanishes on ${\cal R}({\cal A}_n)$, which will then imply uniqueness of $\alpha_{{\cal A}_n}\in R_{n-4}$ (up to scaling), we argue by contradiction and assume that there is at least a pencil of such polynomials.

\begin{lemma} \label{lem:planes}
If the adjoint is not unique, then there is an $f\in R_{n-4}$ that vanishes on ${\cal R}({\cal A}_n)$ and on all the planes $(2,{\rm I},i)$ and $(2,{\rm II},i)$ in $\partial_a {\cal A}_n$. 
\end{lemma}
\begin{proof}
    Each plane $(2, \rm II,i)$ contains $n-4$ lines in ${\cal R}({\cal A}_n)$, one for each line $Z_jZ_{j+1}$ that meets the plane $Z_{i-1}Z_{i}Z_{i+1}$ outside the three lines $Z_{i-1}Z_{i},Z_{i}Z_{i+1},Z_{i-1}Z_{i+1}$. So if the adjoint is not unique, then, for each $i$, there is an adjoint $f$ that vanishes on the plane $(2, {\rm II},i)$. Indeed, the hyperplane in $R_{n-4}$ consisting of $(n-4)$-forms vanishing on an extra point in $(2,{\rm II},i)$ intersects any pencil of adjoints, and the intersection is $f$. Furthermore, two planes $(2, {\rm II},i)$ and $(2, {\rm II},i+1)$ intersect in a point $Z_iZ_{i+1}$.  This point represents the line $Z_iZ_{i+1}$ which does not intersect any of the lines $Z_jZ_{j+1}, j\not= \{i-1,i,i+1,i+2\}$ (by the generality assumption).  Therefore the point $[Z_iZ_{i+1}]$ does not lie on any of the $n-4$ residual lines in $(2, {\rm II},i+1)$, and so an adjoint $f$ that vanishes on $(2, {\rm II},i)$ must also vanish on $(2, {\rm II},i+1)$.  Inductively, this $f$ vanishes on all the planes $(2, {\rm II},i), i=1,\ldots, n$.

    Now, each plane $(2, {\rm I},i)$ intersect the three planes $(2, {\rm II},i-1),(2, {\rm II},i),(2, {\rm II},i+1)$ in a line distinct from the residual lines in ${\cal R}({\cal A}_n)$ (by the generality assumption). So the adjoint $f$ vanishes on these three lines.  In addition it vanishes on $\binom{n-5}{2}$ points in $(2, \rm I,i)$, the intersection of $(2, \rm I,i)$ with residual conics.  These points represent the lines through $Z_i$ that intersect two skew lines $Z_jZ_{j+1}$ and $Z_kZ_{k+1}$ that do not contain $Z_i$.  So we consider the $n-4$ lines $Z_jZ_{j+1}$ so that $j\notin\{i-2,i-1,i,i+1\}$, and their images under projection from $Z_i$. The projection of these lines is $n-4$ lines in a plane, that meet cyclically at the $n-5$ image points of the $Z_j, j\notin\{i-2,i-1,i,i+1,i+2\}$, and a $(n-4)$-th point $p_i$, the intersection of the images of the two lines $Z_{i-3}Z_{i-2}$ and  $Z_{i+2}Z_{i+3}$.  
    This $(n-4)$-gon in the plane has a unique planar adjoint curve of degree $n-7$ that vanishes on its residual $\binom{n-4}{2}-(n-4)=\binom{n-5}{2}-1$ points, and this adjoint does not pass through $p_i$, cf. \cite[Proposition 2.2]{kohn2021adjoints}. 
    
    Therefore the adjoint $f$ of degree $n-4$ that already vanishes on three lines, must vanish on the planar adjoint of degree $n-7$ and the additional point $p_i$.  Then $f$ vanishes necessarily on the plane $(2, {\rm I},i)$.  This argument applies to each plane $(2, {\rm I},i)$. So the lemma follows.
\end{proof}
 \begin{lemma} \label{lem:quadric surfaces}
If the adjoint is not unique, then there is an $f\in R_{n-4}$ that vanishes on ${\cal R}({\cal A}_n)$ and on all the quadric surfaces  $(2,{\rm III},ij)$  in $\partial_a {\cal A}_n$. 
\end{lemma}
\begin{proof}
    By Lemma \ref{lem:planes} we may assume that there is an $f\in R_{n-4}$ that vanishes on 
    ${\cal R}({\cal A}_n)$ and on all the planes of type $(2,{\rm I})$ and $(2,{\rm II})$ in $\partial_a {\cal A}_n$.  Let $S$ be the quadric surface   $(2,{\rm III},ij)$.   It contains a line in each of the four planes $(2,{\rm I},i),(2,{\rm I},i+1),(2,{\rm I},j),(2,{\rm I},j+1)$  and a  line in each of the planes $(2,{\rm II},i),(2,{\rm II},i+1),(2,{\rm II},j),(2,{\rm II},j+1)$ that together form four reducible conics. In addition there are $n-6$ residual conics on $S$;  one for each line $Z_kZ_{k+1}$ that is disjoint from $Z_iZ_{i+1}$ and $Z_jZ_{j+1}$, the lines defining $S$. But then $f \in R_{n-4}$ vanishes on $n-2$ plane sections of $S$, and hence must vanish on all of $S$. Since this argument applies to each quadric surface of type $(2,{\rm II})$, the lemma follows.
\end{proof}

\begin{proposition}\label{prop:unique adjoint} The adjoint $\alpha_{{\cal A}_n}\in R_{n-4}\setminus \{0\}$ is unique, up to scalar.
\end{proposition}

\begin{proof}
We assume that the adjoint is not unique. By Lemmas \ref{lem:planes} and \ref{lem:quadric surfaces} we can choose an adjoint $f$ that vanishes on all the $2$-dimensional strata of $\partial_a {\cal A}_n$. Now fix a boundary quadric threefold $(3,{\rm I},i)$ of 
$\partial_a {\cal A}_n$.  It contains the four planes $(2,{\rm I},i),(2,{\rm I},i+1),(2,{\rm II},i),(2,{\rm II},i+1)$, and the $n-3$ quadric surfaces $(2,{\rm II},ij), j\notin\{i-1,i,i+1\}$.  So $f$ vanishes on a surface of degree $2n-2$ on $(3,{\rm I},i)$.  But $f$ has degree $n-4$, so it must vanish on $(3,{\rm I},i)$.  Again, this argument applies to each of the $n$ boundary quadric threefolds, so $f$ vanishes on a threefold of degree $2n$ in ${\rm Gr}(2,4)$. But then $f$ vanishes on ${\rm Gr}(2,4)$. I.e., it is the zero polynomial in $R_{n-4}$ and the proposition follows.
\end{proof}

\begin{proof}[Proof of Theorem~\ref{adjoint}]
    Propositions~\ref{prop:existence}/\ref{prop:unique adjoint} show existence/uniqueness.    
\end{proof}
\begin{corollary}\label{cor:restricted adjoint}
    The restriction of the adjoint to any quadric threefold $(3,{\rm I},i)$ in $\partial_a {\cal A}_n$ is the unique hypersurface section of degree $n-4$ that interpolates the residual lines and conics contained in the boundary surfaces on the threefold $(3,{\rm I},i)$.
\end{corollary}
\begin{proof} This follows immediately from the proof of uniqueness of $\alpha_{{\cal A}_n}$.
\end{proof}
\begin{remark}
The adjoint $\alpha_{{\cal A}_n}\in R_{n-4}$ does not vanish on any of the boundary vertices of ${\cal A}_n$, it has multiplicity one along the lines and conics in the residual arrangement, and has multiplicity at least one at the residual vertices.  

The adjoint defines a canonical divisor on a $3$-dimensional variety birational to the algebraic boundary $\partial_a {\cal A}_n$.  In fact, the blowup of ${\rm Gr}_{\mathbb{C}}(2,4)$ along the residual arrangement, first along the residual vertices, and then along the strict transforms of the residual curves,  has a canonical divisor
that is equivalent to $-4H+2E_L+2E_C+3E_V$ (cf. \cite[Example 15.4.3]{Fu}), where $E_L$ is the union of the exceptional divisors over the residual lines, $E_C$ is the union of the exceptional divisors over the residual conics and $E_V$ is the union of the exceptional divisors over the residual vertices.  The algebraic boundary is a hypersurface of degree $n$ 
restricted to ${\rm Gr}_{\mathbb{C}}(2,4)$ that has multiplicity three along the residual lines and conics and multiplicity four at the residual vertices.  So, by adjunction, the strict transform $Y_n$ of the algebraic boundary  has a canonical divisor 
that is the restriction of a divisor equivalent to $(n-4)H-E_L-E_C-E_V$ (cf. \cite[Example 4.2.6]{Fu}). The adjoint $\{\alpha_{{\cal A}_n}=0\}$ divisor is therefore the unique divisor whose strict transform is a canonical divisor on $Y_n$. 
\end{remark}

\section{The amplituhedron is a positive geometry} \label{sec:6}
In this section, we define a rational $4$-form on ${\cal A}_n$ with simple poles along the algebraic boundary $\partial_a {\cal A}_n$ and describe recursively its (Poincar\'e) residues along the boundary strata.   We define the $4$-form in terms of the adjoint and argue that there is a unique scaling of the adjoint function such that the iterated residues of the $4$-form are $\pm 1$ at all vertices of the amplituhedron $\cA_n$, thus showing that the $4$-form is a canonical form for a positive geometry.

We prove the main result (Theorem \ref{thm:main}), which we reiterate here.
\begin{theorem*}
    Under Assumption~\ref{assum:general}, the amplituhedron $({\rm Gr}_{\mathbb{C}}(2,4),{\cal A}_n)$ is a positive geometry. 
\end{theorem*}

Like in Sections \ref{sec:4} and \ref{sec:5}, we set ${\cal A}_n = {\cal A}_n(Z)$ for a totally positive matrix $Z$ satisfying Assumption \ref{assum:general}. 
The adjoint $\alpha_{{\cal A}_n}\in R_{n-4}$ defines the zeros of a rational $4$-form on ${\rm Gr}_{\mathbb{C}}(2,4)$ with simple poles along the $\partial {\cal A}_n$ as follows.  We fix a general line $L$ disjoint from all the lines $Z_iZ_j$ in $\PP^3$ and let $D_L \subset {\rm Gr}_{\mathbb{C}}(2,4)$ be the tangent hyperplane section at $[L]$, so that ${\rm Gr}_{\mathbb{C}}(2,4)\setminus D_L\cong \Af^4$, with local coordinates $x_1,x_2,x_3,x_4$. In these coordinates, we can write the restriction of the $4$-form to $U_{L} = {\rm Gr}_{\mathbb{C}}(2,4)\setminus D_L$ as 
\begin{equation} \label{eq:canform}
\Omega({\cal A}_n)_{|U_L} \, = \, \frac{\alpha_{{\cal A}_n}}{\prod_{1\leq i\leq n} \langle AB i(i+1) \rangle} \, d\omega \end{equation}
where $d\omega=dx_1\wedge dx_2\wedge dx_3\wedge dx_4$.
To see that this form is, up to scalar multiple, a canonical form for a positive geometry, we will check the recursive axioms from Definition \ref{def:posgeom}. First we describe the restriction of the residual arrangement ${\cal R}({\cal A}_n)$ to surface and threefold boundary components in $\partial_a{\cal A}_n$.

 \begin{proposition}\label{cor:restricted residuals} In a  plane of type $(2,{\rm I})$, the boundary face is an $(n-1)$-gon and the restriction of the residual arrangement ${\cal R}(\cA_n)$ is the residual arrangement to this $(n-1)$-gon. In a plane of type $(2,{\rm II})$ the boundary face is a triangle and the restriction of ${\cal R}(\cA_n)$  consists of $n-4$ residual lines. In a boundary quadric surface of type $(2,{\rm III})$, the boundary face is a quadrilateral and the restriction of ${\cal R}(\cA_n)$  consists of $n-4$ plane sections.  
    
In a quadric threefold $(3,{\rm I},i)$, the boundary of its facet in $\partial \cA_n$ is the union of the boundary faces in the surfaces of type $(2,{\rm I})$,$(2,{\rm II})$ and $(2,{\rm III})$ contained in the threefold.
The restriction to the threefold of the residual arrangement ${\cal R}(\cA_n)$ is the union of the restrictions to the boundary surfaces. 
\end{proposition}

\begin{proof}
    The boundary plane $H = (2,{\rm I},i)$ contains no residual lines or conics by Figure \ref{fig:rough stratification}. 
    The boundary lines in $H$ are the lines $(1,{\rm II},ij), i\notin \{j,j+1\}$ in addition to the line $(1,{\rm I},i)$.  They define the boundary of the face $H\cap \cA_n$, an $(n-1)$-gon with vertices $(0,{\rm I},ij), j\neq i$ (see the pentagons in Figure \ref{fig:schlegel3D}).  
    A residual vertex in $H$ is of type $(0,{\rm III},ij)$ or $(0,{\rm IV},ijk)$. The vertex $(0,{\rm III},ij)$ is the intersection of the line $(1,{\rm I},i)$ and $(1,{\rm II},ij)$, while the vertex $(0,{\rm IV},ijk)$ is the intersection of $(1,{\rm II},ij)$ and $(1,{\rm II},ik)$.  Altogether these residual vertices are the intersections of boundary lines that do not intersect on the boundary of the $(n-1)$-gon in the plane. For $n = 6$ (Figure \ref{fig:schlegel4D}, right), the situation is as in Figure \ref{fig:adjpentagon}, where $P = H\cap\partial\cA_n$ and ${\cal R}({\cal A}_6) \cap H$ consists of the blue points.
    We conclude that the restriction ${\cal R}(\cA_n) \cap H$ is the residual arrangement in the algebraic boundary of the $(n-1)$-gon, the face $H\cap\partial\cA_n$.
 
   The boundary faces and facets of $\partial\cA_n$ were all described in Proposition \ref{prop:boundary faces}, so it remains to describe the restriction of the residual arrangement.
   
   The plane $H=(2,{\rm II},i)$ contains the $n-4$ residual lines $(1,{\rm III},ij)$. In Figure \ref{fig:boundarystrata}, for the plane $(2,{\rm II},8)$, these are obtained by placing the red tick in the diagram for $(1,{\rm III},28)$ on any of the $n-4 = 4$ edges $(j,j+1)$ not adjacent to the green triangle. Any residual vertex in the plane $H$ is of type $(0,{\rm II}),(0,{\rm III}), (0, {\rm IV})$ or $(0,{\rm V})$, see Figure \ref{fig:rough stratification}.  More precisely, they are the vertices $(0,{\rm II},ij),(0,{\rm III},ij), (0,{\rm IV},(i-1)ik),(0,{\rm IV},(i+1)(i-1)k)$ and $(0,{\rm V},ijk)$.  These residual vertices lie in the residual lines $(1,{\rm III},ij)$, which all lie in $H$. So the restriction ${\cal R}(\cA_n) \cap H$ equals the $n-4$ residual lines.
    
    Similarly, the boundary quadric surface $Q=(2,{\rm III},ij)$ contains $n-4$ plane sections that belong to the residual arrangement: When $|j-i|>2$, then the $n-6$ residual conics $(1,{\rm IV},ijk)$ for $k\notin \{i-1,i,i+1,j-1,j,j+1\}$ all lie on $Q$. Furthermore the four residual lines $(1,{\rm III},ij),(1,{\rm III},(i+1)j), (1,{\rm III},ji), (1,{\rm III},(j+1)i)$ also all lie in $Q$.  Notice that  $(1,{\rm III},ij)\cup (1,{\rm III},(i+1)j)$ and $(1,{\rm III},ji)\cup (1,{\rm III},(j+1)i)$ are plane sections of $Q$.  Now, no other residual lines or conics lie on $Q$, and residual vertices on $Q$ are of type
    $(0,{\rm II})$, $(0,{\rm III})$, $(0,{\rm IV}),(0,{\rm V})$ or $(0,{\rm VI})$ by Figure \ref{fig:rough stratification}.  As above, one may check that each of these residual vertices lies in one of the residual lines or conics on $Q$.
     We conclude that the restriction of ${\cal R}(\cA_n)$ to $Q=(2,{\rm III},ij)$ is $n-4$ plane sections when $|j-i|>2$.
     When $|j-i|=2$, e.g. $j=i+2$, then the $n-5$ residual conics $(1,{\rm IV},i(i+2)k), i+3<k<n+i-1$ all lie on $Q$. Furthermore, the two residual lines $(1,{\rm III},i(i+2))$ and $(1,{\rm III},(i+3)i)$ also all lie in $Q$.  Their intersection is the residual vertex $(0,{\rm II},i(i+3))$, and their union $(1,{\rm III},i(i+2)) \cup (1,{\rm III},(i+3)i)$ forms a plane section of $Q$.  Now, no other residual lines or conics lie on $Q$, and the residual vertices on $Q$ are vertices of type
    $(0,{\rm II})$, $(0,{\rm III})$, $(0,{\rm IV}),(0,{\rm V})$ or $(0,{\rm VI})$ by Figure \ref{fig:rough stratification}.  Again, one may check that each of these residual vertices lies in one of the residual lines or conics on $Q$.
     We conclude that the restriction of ${\cal R}(\cA_n)$ to $Q=(2,{\rm III},i(i+2))$ consists of $n-4$ plane sections.

Any boundary quadric threefold $D$, say $(3,{\rm I},i)$, intersects the algebraic boundary $\partial_a \cA_n $ only along boundary surfaces,  so the residual vertices, lines and conics in $D$ form, by the first part of the proposition, the restriction of the residual arrangement to $D$.       
\end{proof}

For the purpose of our argument, we generalize the definition of residual arrangement so that it applies to boundaries of ${\cal A}_n$. 
Let ${\cal X}_{\geq 0}$ be a semi-algebraic subset of the real points ${\cal X}(\mathbb{R})$ of an irreducible complex variety ${\cal X}$. The Euclidean interior of ${\cal X}_{\geq 0}$ is ${\cal X}_{>0}$. Assume that the algebraic boundary $\partial_a {\cal X}_{\geq 0}$, i.e., the Zariski closure of ${\cal X}_{\geq 0} \setminus {\cal X}_{>0}$ in ${\cal X}$, is a union of irreducible normal divisors $D_1,...,D_r$ with pairwise transverse intersection. The~residual arrangement ${\cal R}({\cal X}_{\geq 0})$ is the union of all intersections of finitely many $D_i$ which do not intersect ${\cal X}_{\geq 0}$. For $({\cal X}, {\cal X}_{\geq 0}) = ({\rm Gr}_{\mathbb{C}}(2,4), {\cal A}_n)$, ${\cal R}({\cal X}_{\geq 0})$ is the residual arrangement from Definition \ref{def:resarr}.

\begin{proposition} \label{prop:residueshaverightdivisors}
The successive residues of the nonzero rational $4$-form $\Omega({\cal A}_n)$ from \eqref{eq:canform} along each closed stratum $D$ in $\partial_a{\cal A}_n$ are the unique, up to scalar multiple, rational forms with simple poles along $\partial_a D_{\geq 0}$ and zeros along the residual arrangement ${\cal R}( D_{\geq 0})$.
\end{proposition}
\begin{proof}
First of all, note that the statement makes sense because the iterated residues of $\Omega({\cal A}_n)$ along closed strata $D$ are defined up to sign (see Remark \ref{rem:sign}).

Consider a quadric threefold $D_i=(3,I,i)\subset \partial_a \cA_n$ and its facet $D_{i,\geq 0}=D_i\cap \cA_n.$ The algebraic boundary $\partial_a D_{\geq 0}$ of $D_{i,\geq 0}$ is the union of boundary surfaces in $D_i$, which by Proposition \ref{cor:restricted residuals} is a union of $n-1$ hyperplane sections. The residual arrangement ${\cal R}(D_{i,\geq 0})$ is clearly contained in the restriction ${\cal R}(\cA_n) \cap D_i$. In fact, by Proposition \ref{cor:restricted residuals}, they coincide.   An adjoint to $\partial_a D_{i,\geq 0}$ is a hypersurface section of degree $n-4$ that interpolates the residual arrangement ${\cal R}(D_{i,\geq 0}).$ By Corollary \ref{cor:restricted adjoint}, this adjoint is unique. We let $\Omega(D_{i,\geq 0})$ be the unique, up to scalar, rational $3$-form with simple poles along $\partial_a D_{i,\geq 0}$ and zeros on the~adjoint. 

To prove the proposition for $D_i$, we must compare $\Omega(D_{i,\geq 0})$ to the residue of $\Omega({\cal A}_n)$ along $D_i$. Since the denominator of $\Omega({\cal A}_n)$ is simple along all its components, the residue along $D_i$ 
is a rational $3$-form whose numerator is a nonzero multiple of the restriction of $\alpha_{{\cal A}_n}$ to $D_i$. In particular, this residue form vanishes on ${\cal R}(D_{i, \geq 0})$. The denominator is the restriction of
$$\prod_{j\not=i} \langle AB j(j+1) \rangle$$
to  $D_{i}$, which cuts out the boundary surface components on the quadric threefold $D_{i}$. By Proposition \ref{prop:boundary faces}, these boundary surface components define a hypersurface section of $D_{i}$ of degree $n-1$. 
Therefore ${\rm Res}_{D_i}\Omega({\cal A}_n)$ coincides, up to scalar multiple, with $\Omega(D_{i,\geq 0})$.

It remains to consider boundary surfaces and boundary curves. 
The boundary surfaces are planes of type $(2,{\rm I})$ or $(2,{\rm II})$ or quadric surfaces of type $(2,{\rm III})$.

In a boundary plane $D \simeq \mathbb{CP}^2$ of type $(2,{\rm I})$,  the face $D_{\geq 0}$ is an $(n-1)$-gon in $\partial{\cal A}_n$, and $\Omega(D_{\geq 0})$ is the canonical form of the positive geometry $(D,D_{\geq 0})$, as in Example \ref{ex:canformpentagon}. 
It has simple poles along the $n-1$ lines and zeros on the adjoint to the $(n-1)$-gon, the unique curve of degree $n-4$ in the plane that interpolates the residual arrangement of the $(n-1)$-gon.

On the other hand, the iterated residue of $\Omega(\cA_n)$ has simple poles along the $n-1$ boundary lines and zeros along the restriction of the adjoint $\{\alpha_{\cA_n}=0\}$, a curve of degree $n-4$ interpolating the restriction of the residual arrangement. This is the residual arrangement of the $(n-1)$-gon (Proposition \ref{cor:restricted residuals}).  We conclude that the iterated residue of $\Omega(\cA_n)$ coincides up to scalar with the canonical $2$-form $\Omega(D_{\geq 0})$ of the positive geometry $(D,D_{\geq 0})$.

Each boundary plane $D$ of type $(2,{\rm II})$ contains a boundary triangle $D_{\geq 0}$ in $\partial{\cal A}_n$ and $n-4$ lines in ${\cal R}({\cal A}_n) \cap D$, by Proposition \ref{cor:restricted residuals}.
The pair $(D,D_{\geq 0})$ is a positive geometry with canonical form $\Omega(D_{\geq 0})$ having simple poles along the lines of the triangle and no zeros.

The iterated residue of $\Omega({\cal A}_n)$ is a $2$-form with poles along the boundary lines in a addition to the $n-4$ lines in the residual arrangement, and zeros along the same $n-4$ lines. After cancellation, this residue coincides, up to scalar, with the canonical form $\Omega(D_{\geq 0})$.

A quadric $(2,{\rm III})$ boundary surface $D$ contains a quadrilateral $D_{\geq 0}=D\cap \partial \cA_n$.  The algebraic boundary $\partial_a D_{\geq 0}$ has no residual arrangement, and the adjoint to the boundary $\partial_a D_{\geq 0}$ of a quadrilateral is a constant, so we let $\Omega(D_{\geq 0})$ be the rational $2$-form with simple poles along $\partial_a D_{\geq 0}$ and no zeros. Notice that $\partial_a D_{\geq 0}$ is the union of two plane sections of $D$.

 The quadric boundary surface $D$ contains $n-4$ plane sections that are curves in the residual arrangement by Proposition \ref{cor:restricted residuals}. This constitutes the vanishing locus of the adjoint $\alpha_{{\cal A}_n}$ in $D$. Together with the two plane boundary curves they form  the intersection of the quadric surface with the quadric boundary threefolds that do not contain the surface.  So the iterated residue of $\Omega({\cal A}_n)$ along $D$ has $n-4$ linear factors in the denominator that cancel all $n-4$ linear factors in the numerator. It thus coincides, up to scalar, with the $2$-form~$\Omega(D_{\geq 0})$.  
 
The boundary curves in $\partial_a{\cal A}_n$ are all lines of type $(1,{\rm I})$ or $(1,{\rm II})$.  Each line $D$ contains a line segment $D_{\geq 0}=D\cap\partial \cA_n$ with two boundary vertices by Proposition \ref{prop:boundary faces} and $n-4$ residual vertices in ${\cal R}({\cal A}_n)$. As in Example \ref{ex:canformlinesegment}, $(D,D_{\geq 0})$ is a positive geometry with a canonical form $\Omega(D_{\geq 0})$ with simple poles at the two boundary vertices and no zeros.

The boundary line $D=(1,{\rm I},i)$ is contained in the two boundary threefolds $(3,{\rm I},i-1)$ and $(3,{\rm I},i)$, the boundary vertices are the intersections with the boundary threefolds $(3,{\rm I},i-2)$ and $(3,{\rm I},i+1)$ and the residual vertices are the intersections with the remaining $n-4$ boundary threefolds.  
A boundary line $D=(1,{\rm II},ij)$, is contained in  the three boundary threefolds $(3,{\rm I},i-1),(3,{\rm I},i)$ and $(3,{\rm I},j)$. The boundary vertices are the intersections with the boundary threefolds $(3,{\rm I},j-1)$ and $(3,{\rm I},j+1)$ and the residual vertices are the intersections with the remaining $n-5$ boundary threefolds in addition to the vertex $(0,{\rm III},ij)$.

The adjoint $\alpha_{{\cal A}_n}$ vanishes on the $n-4$ residual vertices on $D$, so whether $D$ is of type $(1,{\rm I})$ or $(1,{\rm II})$, in the iterated residue $1$-form of $\Omega(\cA_n)$ along the line a cancellation leaves a rational form with simple poles in the boundary vertices and no zeros. This form must therefore coincide, up to scalar, with the canonical form $\Omega(D_{\geq 0})$.
\end{proof}

\begin{proof}[Proof of Theorem~\ref{thm:main}]
Note that the only candidates for the canonical form of ${\cal A}_n$ are scalar multiples of the form $\Omega({\cal A}_n)$ from \eqref{eq:canform}. Indeed, the denominator of \eqref{eq:canform} is fixed by the requirement that the canonical form has simple poles along the algebraic boundary (Proposition \ref{prop:algboundary}). From our proof of Proposition \ref{prop:residueshaverightdivisors}, it follows that the canonical form, if it exists, has zeros along the residual arrangement ${\cal R}({\cal A}_n)$. This uniquely defines the numerator of the rational function in \eqref{eq:canform} by Theorem \ref{adjoint}.

To complete the proof of Theorem \ref{thm:main}, we fix a non-zero scalar for the $4$-form $\Omega({\cal A}_n)$, and  choose a boundary vertex $v$, a boundary line $L_v$ through $v$, a boundary surface $S_v$ containing the boundary line $L_v$ and a boundary threefold $D_v$ containing $S_v$. An orientation on ${\cal A}_n$ induces successively an orientation on the flag of boundary faces, 
$$D_v\supset S_v\supset L_v\ni v$$
at $v$, and hence unique iterated residues of $\Omega({\cal A}_n)$ at each face depending on the flag.  In particular, the residue on $L_v$ is a $1$-form with a simple pole at $v$ and no zeros, and hence a nonzero residue $\lambda_v\in \C$ at $v$.  Thus, we may choose the scalar for $\Omega({\cal A}_n)$ such that the iterated residue $\lambda_v$ equals $1$ at $v$.  
The iterated residue at the other boundary vertex on $L_v$ is then $-1$.
Similarly, by the connectedness of the $1$-skeleton of the boundary, see Corollary \ref{cor:connected one-skeleton}, the iterated residue at any other boundary vertex is $\pm 1$, with sign depending on the flag of boundary faces at the vertex.  
We conclude that we can choose a scalar such that the $4$-form $\Omega({\cal A}_n)$ is a canonical form for $({\rm Gr}_{\mathbb{C}}(2,4),{\cal A}_n)$ as a positive geometry.
\end{proof}

\section*{Acknowledgements}
We are grateful to Nima Arkani-Hamed, Paolo Benincasa and Johannes Henn for suggesting to study the adjoint of the amplituhedron. We want to thank Thomas Lam, Matteo Parisi and Jaroslav Trnka for useful conversations.  

\bibliographystyle{abbrv}
\bibliography{references2.bib}

\noindent{\bf Authors' addresses:}
\medskip

\noindent Kristian Ranestad, University of Oslo
\hfill {\tt ranestad@math.uio.no}

\noindent Rainer Sinn, University of Leipzig
\hfill {\tt rainer.sinn@uni-leipzig.de}

\noindent Simon Telen, MPI-MiS Leipzig
\hfill {\tt simon.telen@mis.mpg.de}
\end{document}